\documentclass{amsart}
\usepackage{amsfonts}
\usepackage[T1]{fontenc}
\usepackage{url}
\usepackage{xy}
\input xy
\usepackage{cleveref}
\xyoption{all}
\usepackage{amsmath}
\usepackage{amssymb}
\usepackage{amsthm}
\usepackage{geometry}
\usepackage{verbatim}
\newcommand{\HC}{\mathrm{HC}}
\newcommand{\acrysh}{\widehat{\acrys}}
\newcommand{\mot}{\mathrm{mot}}
\newcommand{\spec}{\mathrm{Spec}}
\newcommand{\shv}{\mathrm{Shv}}

\renewcommand{\phi}{\varphi}

\newcommand{\D}{\mathcal{D}}
\newcommand{\HP}{\mathrm{HP}}
\newcommand{\TC}{\mathrm{TC}}

\newcommand{\TP}{\mathrm{TP}}
\newcommand{\sF}{\mathcal{F}}
\newcommand{\qrsp}{\mathrm{QRSPerfd}}
\newcommand{\qrspp}{\mathrm{QRSPerf}_{\mathbb{F}_p}}

\newcommand{\triplearrows}{\begin{smallmatrix} \to \\ \to \\ 
\to \end{smallmatrix} }

\usepackage{relsize}
\usepackage[bbgreekl]{mathbbol}
\usepackage{amsfonts}
\DeclareSymbolFontAlphabet{\mathbb}{AMSb} 
\DeclareSymbolFontAlphabet{\mathbbl}{bbold}
\newcommand{\Prism}{{\mathlarger{\mathbbl{\Delta}}}}
\newcommand{\Prismc}{\widehat{\Prism}}

\theoremstyle{definition}
\newtheorem{definition}{Definition}[section]
\newtheorem{question}[definition]{Question}
\newcommand{\gr}{\mathrm{gr}}
\newtheorem{variant}[definition]{Variant}
\newtheorem{construction}[definition]{Construction}
\newtheorem{example}[definition]{Example}
\newtheorem{remark}[definition]{Remark}

\theoremstyle{theorem}
\newtheorem{corollary}[definition]{Corollary}
\newtheorem{theorem}[definition]{Theorem}
\newtheorem{proposition}[definition]{Proposition}

\makeatletter
\@namedef{subjclassname@2020}{\textup{2020} Mathematics Subject Classification}
\makeatother
\begin{document}

\renewcommand{\sp}{\mathrm{Sp}}
\newcommand{\THH}{\mathrm{THH}}

\title{Some recent advances in topological Hochschild homology}
\author{Akhil Mathew}
\address{Department of Mathematics, University of Chicago, 5734 S University
Ave., Chicago, IL 60637 USA}
\email{amathew@math.uchicago.edu}

\date{\today}

\maketitle

\begin{abstract}
We give an account of the construction of the Bhatt--Morrow--Scholze motivic
filtration on topological cyclic homology and related invariants, focusing on
the case of equal characteristic $p$ and the connections to crystalline and de
Rham--Witt theory. 
\end{abstract}

\section{Introduction}

Let $X$ be a smooth quasi-projective scheme over a field $k$. 
In this case, one has the algebraic $K$-theory spectrum $K(X)$ of $X$,
defined by Quillen \cite{Qui72} using the exact category of vector bundles on $X$ (in
general, one should  use perfect complexes as in \cite{TT90}). One can think of $K(X)$ as a type of ``cohomology theory'' for the scheme $X$,
analogously to the topological $K$-theory of a compact topological space. 
With this in mind, 
the following fundamental result gives an analog of the classical Atiyah--Hirzebruch
spectral sequence relating topological $K$-theory to singular cohomology. 

\begin{theorem}[The motivic filtration on algebraic $K$-theory,
\cite{FS00, Lev08}] 
\label{motivicfilt}
There is a functorial, convergent,  decreasing  
multiplicative filtration 
$\mathrm{Fil}^{\geq \ast} K(X)$ and identifications
$\mathrm{gr}^i K(X) \simeq \mathbb{Z}(i)^{\mot}(X)[2i]$ for $i \geq 0$. \end{theorem} 

Here the $\mathbb{Z}(i)^{\mot}(X)$, called the \emph{motivic cohomology} of $X$, are 
explicit cochain complexes introduced by Bloch \cite{Blo86} 
(see also \cite{Voe02})
in terms of
algebraic cycles on $X \times \mathbb{A}^n_k$ for $n \geq 0$. 
In particular, we have that $$H^{2i}(X, \mathbb{Z}(i))
\stackrel{\mathrm{def}}{=} 
H^{2i}( \mathbb{Z}(i)^{\mot}(X))
= \mathrm{CH}^i(X)$$ is given by the Chow
group $\mathrm{CH}^i(X)$ of codimension $i$ cycles on $X$ modulo rational
equivalence.\footnote{One difference in this analogy is that algebraic
$K$-theory historically preceded motivic cohomology, whereas singular
cohomology preceded topological $K$-theory.} 
The complexes $\mathbb{Z}(i)^{\mot}(X)$ (considered as objects of the derived category
of abelian groups) can also be described as maps in the $\mathbb{A}^1$-motivic stable homotopy
category from $X$ into motivic Eilenberg--MacLane spectra.

\Cref{motivicfilt} gives substantial information about algebraic $K$-theory,  
especially after profinite completion. After reducing modulo a prime $l$ which is
different from the
characteristic, the Beilinson--Lichtenbaum conjecture proved by
Voevodsky--Rost 
\cite{Voe03, Voe11}
identifies mod $l$ motivic cohomology as Zariski
(or Nisnevich) sheaves,
\begin{equation}   \label{BLconj} \mathbb{Z}/l(i)^{\mathrm{mot}} \simeq \tau^{\leq i} (R \nu_*  \mu_l^{\otimes
i} ), \end{equation}
for $\nu$ 
the pushforward from the \'etale to the Zariski topology. 
For example, 
\Cref{motivicfilt}
implies that in high degrees, we can compute mod $l$ algebraic $K$-theory
of a variety over $\mathbb{C}$ as the topological $K$-theory of the space of 
$\mathbb{C}$-points and is therefore finitely generated since the
underlying homotopy type
is that of a finite CW complex. (By contrast, $K_0$ with mod $l$ coefficients can be
enormous, cf.~for instance \cite{Sch02, RS10, Tot16}.) 

After $p$-adic completion when the ground field $k$ has characteristic $p$, 
the analog of the 
Beilinson--Lichtenbaum conjecture \eqref{BLconj} is given by the theorems of
Geisser--Levine \cite{GL00} and Bloch--Kato--Gabber \cite{BK}, 
\[ \mathbb{Z}/p(i)^{\mathrm{mot}} \simeq \Omega^i_{\mathrm{log}}[-i],  \]
identifying the object $\mathbb{Z}/p(i)^{\mathrm{mot}}$ (which lives in the
derived category of Zariski or Nisnevich sheaves on $X$) with the $-i$-shift of
the 
subsheaf $\Omega^{i}_{\mathrm{log}} \subset \Omega^i$ of differential $i$-forms
generated by $\frac{dx_1}{x_1} \wedge \dots \wedge \frac{dx_i}{x_i}$, for the
$x_i$ local units. 
For example, this coupled with 
\Cref{motivicfilt}
implies that if $X$ is a smooth variety over a perfect field
of characteristic $p$, then the mod $p$ $K$-theory $K_*(X; \mathbb{Z}/p)$ 
vanishes in degrees $>\dim(X)$.

However, the construction of \Cref{motivicfilt} (and the definition of motivic
cohomology) relies heavily on the
smoothness assumption on $X$. 
The higher Chow groups of a singular variety $X$ over a field $k$ give an analog
of Borel--Moore homology rather than cohomology; in particular, they are
nilinvariant (while algebraic $K$-theory is far from nilinvariant) and lack a product structure. 
The existing constructions of 
the motivic filtration use the $\mathbb{A}^1$-invariance of $K$-theory (valid
only in the regular case), and giving a general notion of the
motivic filtration (or of motivic
cohomology) on $K$-theory applicable to singular rings appears to be an open
problem.  
Even the setting of regular rings in mixed characteristic is not fully
understood (but see \cite{Lev06}). 

 \begin{question} 
Is there a motivic 
filtration on $K(R)$ for any ring $R$ 
which extends \Cref{motivicfilt} when $R$ is smooth over a field?\footnote{For
affine schemes, one
proposal is to left Kan extend motivic cohomology from smooth algebras using the
result of Bhatt--Lurie 
(cf.~\cite[Appendix A]{EHKSY} for an account)
that the connective $K$-theory of rings is left Kan
extended from smooth algebras. However, this construction will not satisfy
Zariski descent and does not obviously globalize; moreover, it does not apply to
nonconnective $K$-theory.} 
\end{question}

In the study of the algebraic $K$-theory of singular rings, 
the main new tool is the theory of trace methods. 
Trace methods provide maps from algebraic $K$-theory to more
computable invariants built from Hochschild homology, and the basic tools are
relative comparison results to the effect that the homotopy fiber of such maps
satisfy excision and nilinvariance.  
The most general and powerful form of these results uses topological cyclic
homology $\TC$, 
introduced by B\"okstedt--Hsiang--Madsen \cite{BHM93} in the $p$-complete
case (see \cite{DGM13} for the
integral version), which compares to $K$-theory via the
cyclotomic trace map  
\begin{equation} 
\label{cyctrace}
K(R) \to \TC(R). 
\end{equation}

\begin{theorem} 
The homotopy fiber $F$ of 
\eqref{cyctrace} has the following properties: 
\begin{enumerate}
\item  (Dundas--Goodwillie--McCarthy \cite{DGM13}\footnote{See also
\cite{Raskin} for an account using the approach to $\TC$ of \cite{NS18}.}) $F$ is nilinvariant. 
\item (Land--Tamme \cite{LT19}) $F$ satisfies excision, i.e., given a pullback
square of rings with the vertical arrows surjective, then $F$ carries this to a
pullback of spectra. 
\item (Clausen--Mathew--Morrow \cite{CMM}) The profinite completion of the
variant $F'(R) = \mathrm{fib}(K_{\geq 0}(R) \to \TC(R))$  is rigid 
for henselian pairs, i.e., if $(R, I)$ is a henselian pair, then $F'(R)/n
\xrightarrow{\sim} F'(R/I)/n$ for any integer $n > 0$. 
\end{enumerate}
\end{theorem} 

When $R$ is a $\mathbb{Q}$-algebra, then $\mathrm{TC}(R)$ agrees with the
negative cyclic homology $\mathrm{HC}^-(R/\mathbb{Q})$, which is closely related
to the  de Rham cohomology of $R$, and parts (1) and (2) of the result are due
respectively to Goodwillie \cite{Goo86} and to Corti\~nas \cite{Cort06}. 
For example, compare \cite{GRW89} for applications of these results (at the
time conjectural) to the calculation of the $K$-theory of singular
curves over $\mathbb{Q}$ (e.g., rings such as $\mathbb{Q}[x,y]/(xy)$).  
Part (3) in this case, or more generally when $n$  is invertible on $R$,  is the Gabber rigidity
theorem \cite{Gab92}.

In this survey, we concentrate on the situation after $p$-adic completion for $p$-adic
rings, in which case parts (1) and (2) are due to McCarthy \cite{McCarthy97} and
Geisser--Hesselholt \cite{GH06birelative}. 
In this case, it is known that the 
map \eqref{cyctrace} is not only 
useful for detecting ``infinitesimal'' behavior, but 
is also an 
absolute approximation to $p$-adic $K$-theory $R$: specifically, it is $p$-adic
\'etale (connective) $K$-theory;  moreover,  it is an equivalence in high enough degrees
depending on the ring, under mild hypotheses. 
This follows from the work of Geisser--Levine \cite{GL00} on the $p$-adic
$K$-theory of smooth algebras in characteristic $p$, Geisser--Hesselholt
\cite{GH99} on
$\mathrm{TC}$ of such rings, extended to the more general situation using 
the rigidity result of \cite{CMM} (see also \cite{CM19}). 
\begin{theorem}  
\begin{enumerate}
\item  
For $p$-adic rings $R$, 
the trace \eqref{cyctrace} exhibits $\TC(R; \mathbb{Z}_p)$ as the $p$-completion
of the \'etale $K$-theory of $R$.\footnote{Strictly speaking, this refers to the
\'etale sheafification of the connective $K$-theory of $R$.} 
\item 
If $R$ is $p$-complete, $R/p$ has finite Krull dimension 
and $d = \sup_{x \in \spec(R/p)} \mathrm{log}_p [ k(x): k(x)^p]$, then 
the map $K(R; \mathbb{Z}_p) \to \TC(R; \mathbb{Z}_p)$ is an equivalence 
in degrees $\geq \mathrm{max}(d, 1)$. 
\end{enumerate}
\end{theorem} 

The definition of $\TC$ is much more elaborate than that of $K$-theory
and (in particular in the $p$-adic setting) requires significant homotopical
foundations, classically approached through equivariant stable homotopy theory
\cite{BM15},
which were recently dramatically reworked (and simplified) by Nikolaus--Scholze
\cite{NS18} (see also \cite{AMGRnaive, BG} for $\infty$-categorical accounts of the theory
of cyclotomic spectra). 
We refer to \cite{HN19} for a modern survey of topological Hochschild
and cyclic homology. 

There is a sense, however, in which topological cyclic homology is a
structurally simpler
theory: while the building blocks of algebraic $K$-theory come from algebraic
cycles, the building blocks for topological cyclic 
homology are the 
cotangent complex and its wedge powers. In practice, this means that the formal
properties of $\TC$ are somewhat simpler (e.g., $\TC$ has much better descent properties), and that $\TC$ is easier to compute. There is an extensive literature calculating various instances of the 
$p$-adic $K$-theory of 
$p$-adic rings using $\TC$, cf.~for instance \cite{HM97, HM97trunc, HM03, HM04,
GH06} for some examples.

The work of Bhatt--Morrow--Scholze \cite{BMS2} constructs analogs of the motivic
filtration for topological Hochschild homology and its variants in great
generality; the associated
graded objects of this filtration are objects of deep interest in arithmetic
geometry and especially $p$-adic Hodge theory, and have now been constructed
purely algebraically using the prismatic theory \cite{Prisms}. 
Here we state first the analog for $\mathrm{TC}$. 

\begin{theorem}[Bhatt--Morrow--Scholze \cite{BMS2}] 
\label{motivicfiltTC}
Let $R$ be any $p$-complete ring.\footnote{The work
\cite{BMS2} only treats the case where $R$ is quasisyntomic; it is shown in
\cite[Sec.~5]{AMNN} that the construction naturally extends (via left Kan
extension) to all $p$-adic rings.} 
Then there exists a natural multiplicative, $\mathbb{Z}_{\geq 0}^{op}$-indexed
convergent filtration $\mathrm{Fil}^{\geq \ast} \mathrm{TC}(R; \mathbb{Z}_p)$
with associated graded terms $\mathrm{gr}^{i} \mathrm{TC}(R; \mathbb{Z}_p)
\simeq \mathbb{Z}_p(i)(R)[2i]$, for the $\mathbb{Z}_p(i)(R)$ natural objects
of the $p$-complete derived $\infty$-category.  The constructions
$\mathbb{Z}_p(i)$ (as functors to the derived $\infty$-category
$\mathcal{D}(\mathbb{Z}_p)$) satisfy
flat descent, and for regular $\mathbb{F}_p$-algebras  reproduce the objects $R
\Gamma_{\mathrm{proet}}( -, W \Omega^i_{\mathrm{log}})[-i]$.  
\end{theorem}

\Cref{motivicfiltTC} is supposed to be an analog
of \Cref{motivicfilt} for $\mathrm{TC}$. It is not entirely clear  if this
analogy can be made precise (i.e., if both filtrations can be realized as
instances of a common construction). However, it is at least known in the case where
$R$ is a smooth algebra over a field $k$ of characteristic $p$ (so that
\Cref{motivicfilt} is in effect), the 
 filtration $\mathrm{TC}(R; \mathbb{Z}_p)$ is the $p$-completion of the \'etale
 sheafification of the motivic filtration on $K(R)$ (indeed, both are Postnikov
 towers in the (pro-)Nisnevich and (pro-)\'etale topologies). 
The construction of \Cref{motivicfiltTC} has the advantage of being very direct: the filtration is the Postnikov filtration when
these invariants are considered as sheaves in the quasisyntomic topology
(\Cref{sec:quasisyntomic}).

The $\mathbb{Z}_p(i)$ in their most generality are supposed to be a 
general version of $p$-adic \'etale motivic cohomology for $p$-adic rings. 
They arise as a type of 
filtered Frobenius eigenspaces on prismatic cohomology, a new $p$-adic
cohomology theory for $p$-adic formal schemes introduced by Bhatt--Scholze
\cite{Prisms} of deep interest in integral $p$-adic Hodge theory and constructed in some cases 
in \cite{BMS1, BMS2}. 
The case of ``absolute'' prismatic cohomology 
was originally 
constructed using topological Hochschild homology. 
For the formulation of the next result, we write $\THH$ for topological
Hochschild homology equipped with its natural $\mathbb{T}$-action. 

\begin{theorem}[{\cite{BMS2}}]
Let $R$ be a formally smooth algebra over a perfectoid ring $R_0$. 
Then there is a 
complete, exhaustive $\mathbb{Z}$-indexed filtration on $\TP(R; \mathbb{Z}_p)
\stackrel{\mathrm{def}}{=} \THH(R; \mathbb{Z}_p)^{t\mathbb{T}}$ such that 
$\mathrm{gr}^i \mathrm{TP}(R; \mathbb{Z}_p) \simeq \Prismc_{R}[2i]$. 
\end{theorem}

In fact, the above filtration is constructed using the similar descent
techniques, which gives a construction of prismatic cohomology, independent of
the prismatic site of \cite{Prisms}. When $R_0$ is the ring of integers
$\mathcal{O}_C$ in a
complete, algebraically closed 
nonarchimedean field $C$, then the $\Prismc_R$ recover the
${A}_{\mathrm{inf}}$-cohomology of \cite{BMS1}: in particular, they
specialize both to the de Rham cohomology of the formal scheme
$\mathrm{Spf}(R)$ and the $p$-adic \'etale cohomology of the generic fiber. 

We will not attempt to do justice to the new landscapes of integral $p$-adic
Hodge theory. 
In this survey article, we will work through the characteristic $p$ situation in
some detail, in particular, constructing the filtration on $\TP(R;
\mathbb{Z}_p)$ for $R$ a smooth (or more generally quasisyntomic, e.g., lci)
$\mathbb{F}_p$-algebra $R$ and identifying the associated graded pieces in
terms of crystalline cohomology. 
In equal characteristic $p$, absolute prismatic cohomology in this context
reduces to crystalline cohomology, constructed using the (quasi-)syntomic site
instead of the crystalline site (an approach that goes back to \cite{FM87}). 
Finally, we will circle back to the motivation of algebraic
$K$-theory, and explain how one can recover the calculations of the algebraic
$K$-theory of the dual numbers over a perfect field \cite{HM97, Speirs}. It would be
interesting to revisit other such calculations. 

\subsection*{Acknowledgments}
It is a pleasure to thank Benjamin Antieau, 
Alexander Beilinson, Bhargav Bhatt, Dustin Clausen, Vladimir Drinfeld, Lars
Hesselholt, Jacob Lurie, Matthew Morrow, Thomas
Nikolaus, Nick Rozenblyum, and Peter Scholze for numerous helpful conversations related to this
subject over the past few years. I thank Benjamin Antieau, Burt Totaro,  H\'el\`ene Esnault,
Jonas McCandless, and
the referee for comments and corrections on an
earlier version. 
I also thank  Lars Hesselholt and Shuji Saito for the invitation to a workshop
on this material in Hara-mura, and 
Mike Hopkins and Jacob Lurie for organizing the Thursday seminar at
Harvard in 2015--2016 on this subject. 
This work was done while the author was a Clay Research Fellow. 

\subsection*{Notation} We let $\mathbb{T}$ denote the circle group. 
We denote by $\sp$ the $\infty$-category of spectra, with the smash product
$\otimes$, and $\mathbb{S}$ the sphere spectrum. 
For a ring $R$, we let $\mathcal{D}(R)$ denote the derived $\infty$-category of
$R$. 

We will freely use the language of higher algebra \cite{HA}, and in particular the theory of
$\mathbb{E}_\infty$-ring spectra. 
We refer to \cite{Gepner} for a modern survey and introduction.

Throughout the paper, we fix a prime $p$. 
We will occasionally use the theory of $\delta$-rings, but only 
in the $p$-torsionfree case; a $p$-torsionfree $\delta$-ring consists of a
commutative ring $R$ equipped with an endomorphism $\varphi: R \to R$ which
lifts the Frobenius modulo $p$. 
We refer to \cite[Sec.~2]{Prisms} for an account of the theory of $\delta$-rings
in general.

We will often drop the notation of $p$-completions, since we will almost
exclusively be working with $p$-complete objects.

\section{Topological Hochschild homology}

Let $R$ be a commutative ring. 

\newcommand{\einf}{\mathbb{E}_\infty}
\newcommand{\TR}{\mathrm{TR}}
\newcommand{\HH}{\mathrm{HH}}
\begin{definition} 
\label{THH:def}
The \emph{topological Hochschild homology} $\THH(R)$ is the 
universal $\einf$-algebra equipped with a $\mathbb{T}$-action and a map $R \to
\THH(R)$. 
As an $\einf$-ring, there is an identification
\[ \THH(R) = R \otimes_{R \otimes_{\mathbb{S}} R} R ,  \]
and $\THH(R)$ can be obtained as the geometric realization of the cyclic bar
construction (a simplicial $\einf$-ring obtained from the tensor powers of $R$
as a spectrum), \cite[Sec.~III.2]{NS18}. 
\end{definition}

\Cref{THH:def} (which works equally for an $\mathbb{E}_\infty$-ring $R$) is not the most flexible definition of $\THH$, since $\THH$ is
more generally defined for stable $\infty$-categories (it is a localizing
invariant in the sense of \cite{BGT13}, like algebraic $K$-theory); then $\THH$ of a
commutative ring is defined as $\THH$ of its $\infty$-category of perfect
complexes. 
For example, $\THH$ can be defined using factorization homology over the circle
\cite{AMGR}. This perspective, while extremely important in the foundations of the
theory (in particular, in producing the cyclotomic trace), plays less of a role
in the work of \cite{BMS2}, which focuses on commutative rings. The above
formulation for $\mathbb{E}_\infty$-rings is due to \cite{MSV}.

Before describing some of the features of $\THH$, we begin by reviewing the
simpler algebraic analog of Hochschild homology.

\begin{variant}[Classical Hochschild homology] 
\label{classicalHH}
Let $R$ be a commutative $k$-algebra, for $k$ a base ring. 
Then the \emph{Hochschild homology} $\HH(R/k)$ is the universal
$\einf$-algebra\footnote{One could also formulate the universal property in
terms of animated (or simplicial) commutative $k$-algebras, without using the
language of $\einf$-rings.}
under $k$ equipped with a $\mathbb{T}$-action
and a map $R \to \HH(R/k)$ of $\einf$-$k$-algebras. 
As an $\einf$-$k$-algebra, one has
\[ \HH(R/k) = R \otimes^L_{R \otimes^L_k R} R .   \]
When $k = \mathbb{Z}$, we will sometimes drop the $k$ in the above notation. 
\end{variant} 

Hochschild homology over $k$ is a very controllable construction, because of the
classical Hochschild--Kostant--Rosenberg theorem. 
Since $\mathbb{T}$ acts on the 
$\einf$-$k$-algebra $\HH(R/k)$, 
one obtains a commutative differential graded algebra structure on $\HH_*(R/k)$
(with the differential arising from the $\mathbb{T}$-action). 
The Hochschild--Kostant--Rosenberg theorem gives a natural isomorphism
of commutative differential graded algebras
for $R$ smooth over $k$,
\begin{equation}  \HH_*(R/k) \simeq  (\Omega^{\ast}_{R/k}, d),  \end{equation}
where $d$ is the de Rham differential.

Topological Hochschild homology is a much richer theory than classical
Hochschild homology for $p$-adic rings (whereas for $\mathbb{Q}$-algebras, it
reduces to Hochschild homology relative to $\mathbb{Q}$). 
Taking Hochschild homology over the base $\mathbb{S}$ leads to extra
symmetries in the theory which are not available with 
an ordinary ring (e.g., $\mathbb{F}_p$ or $\mathbb{Z}$) as the base; moreover,
it leads to B\"okstedt's computation of $\THH(\mathbb{F}_p)$. 
We begin by reviewing these aspects, following Nikolaus--Scholze \cite{NS18}; see also the survey
\cite{HN19} for a more detailed overview. 

\begin{construction}[{The cyclotomic Frobenius on $\THH$, cf.
\cite[Sec.~IV.2]{NS18}}] 
Given a commutative ring $R$, one has a natural $\mathbb{T}$-equivariant map 
\begin{equation} 
\varphi: \THH(R)  \to \THH(R)^{tC_p},
\label{cycfrob}
\end{equation} 
using the natural embedding $C_p \subset \mathbb{T}$ and the $\mathbb{T} \simeq
\mathbb{T}/C_p$-action on the right-hand-side. 
To construct
\eqref{cycfrob}, 
we use the universal property of $\THH(R)$
to construct a map 
of $\mathbb{E}_\infty$-rings 
$R \to \THH(R)^{tC_p}$ and then extend it canonically to a
$\mathbb{T}$-equivariant map as in \eqref{cycfrob}. This in turn comes from the
\emph{Tate diagonal} \cite[Sec.~III.1]{NS18}
\[ R \to (R \otimes_{\mathbb{S}} \dots \otimes_{\mathbb{S}} R)^{tC_p} =(
\otimes^{C_p}_{\mathbb{S}} R )^{tC_p},  \]
(which exists for every spectrum) 
followed by the map 
$(\otimes^{C_p}_{\mathbb{S}} R)^{tC_p} \to \THH(R)^{tC_p}$ obtained 
from the inclusion $C_p \subset \mathbb{T}$. 
\end{construction} 

The map $\varphi$ is called the cyclotomic Frobenius and plays a central role
in the theory.
Its construction depends crucially on working over the sphere spectrum, and is
thus a feature of $\mathrm{THH}$ that does not exist for ordinary Hochschild
homology: by
universal properties of the $\infty$-category of spectra, one shows  
\cite[Sec.~III.1]{NS18}
that
there is a canonical, lax symmetric monoidal natural transformation, called the
\emph{Tate diagonal},
\begin{equation} \label{Tatediagonal} X \to (X \otimes_{\mathbb{S}} X \otimes_{\mathbb{S}} \dots
\otimes_{\mathbb{S}}  X)^{tC_p} \end{equation}
for any spectrum $X$. 
Indeed, the Tate diagonal is roughly analogous to (and refines) the map
defined for every abelian group $A$, 
$$A \to
H^0_{\mathrm{Tate}}(C_p, A^{\otimes p}) = (A^{\otimes p})^{C_p}/ \mathrm{norms}
,$$
given by the 
formula $a \mapsto a \otimes a \otimes \dots \otimes a$. 
The analog of the map \eqref{Tatediagonal} does not exist in $\mathcal{D}(\mathbb{Z})$, the derived
$\infty$-category of the integers $\mathbb{Z}$, which lacks the analogous
universal property, and is a key reason why $\THH$ yields a richer theory. 

In the work  \cite{NS18},
it is shown that $\varphi$ 
is enough to study the so-called ($p$-typical) ``cyclotomic structure'' on the
$p$-completion $\THH(R; \mathbb{Z}_p)$; in particular, it can be used to define
the topological cyclic homology. 

\newcommand{\can}{\mathrm{can}}

\begin{construction}[Topological cyclic homology] 
Let $R$ be a ring (or more generally a connective $\mathbb{E}_\infty$-ring). 
We define $\TC^-(R) = \THH(R)^{h\mathbb{T}} $ and $ \TP(R) = \THH(R)^{t
\mathbb{T}}$ to be the $\mathbb{T}$-homotopy fixed points and Tate
construction, respectively. 
We have two maps
\[ \can, \varphi: \TC^-(R; \mathbb{Z}_p) \to \TP(R; \mathbb{Z}_p),  \]
where $\can$ 
is the canonical map from $\mathbb{T}$-invariants to the $\mathbb{T}$-Tate
construction, and $\varphi$ is obtained by taking $\mathbb{T}$-invariants from
\eqref{cycfrob} and using the identification 
$\TP(R; \mathbb{Z}_p) = \left( \THH(R; \mathbb{Z}_p)^{tC_p} \right)^{h
\mathbb{T}/C_p}$, cf.~\cite[Lem.~II.4.2]{NS18}. 
In particular, since $\can$ identifies $\pi_0 \TC^-(R; \mathbb{Z}_p)$ and
$\pi_0 \TP(R; \mathbb{Z}_p)$, we can regard 
$\varphi$ as an endomorphism of the ring $\pi_0 \TC^-(R; \mathbb{Z}_p)$. 
The spectrum $\TC(R; \mathbb{Z}_p)$ is the homotopy equalizer
\begin{equation} 
\label{formulaTC}
 \TC(R; \mathbb{Z}_p) = \mathrm{fib}( \varphi - \can): \TC^-(R; \mathbb{Z}_p)
 \to \TP(R; \mathbb{Z}_p).
\end{equation} 
\end{construction} 

The expression \eqref{formulaTC} 
is very different from ones that appear in the more classical approach to
$\THH$. 
It plays an essential role in the work \cite{BMS2}, and has many applications both
structural and computational. For example, it implies the following basic
structural feature of $\TC$: for connective ring spectra, the construction
$\TC/p$ commutes with filtered colimits \cite[Theorem G]{CMM}.

By contrast, 
in the classical approach to $\THH$ and cyclotomic spectra via equivariant
stable homotopy theory (cf.~\cite{Ma94} for a survey), 
the objects $\TC^-, \TP$ do not play a direct role. 
One constructs the structure of a genuine $C_{p^n}$-spectrum on $\THH(R)$, which enables one to
form 
various fixed points 
$\THH(R)^{C_{p^n}}, n \geq 0$, together with 
maps
\[  R, F: \THH(R)^{C_{p^n}} \to \THH(R)^{C_{p^{n-1}}},  \quad V:
\THH(R)^{C_{p^{n-1}}} \to \THH(R)^{C_{p^n}}.   \]
The various fixed points are related to each other inductively using the cofiber
sequences
\[  \THH(R)_{hC_{p^n}} \to \THH(R)^{C_{p^n}} \stackrel{R}{\to}
\THH(R)^{C_{p^{n-1}}} . \]
In particular, one forms the inverse limit
\[ \TR(R) = \varprojlim_R \THH(R)^{C_{p^n}},  \]
which is a connective $\mathbb{E}_\infty$-ring. 
The maps $F, V$ act on $\TR(R)$, and $\TC(R; \mathbb{Z}_p)$ is the equalizer
$\mathrm{fib}(F - 1: \TR(R; \mathbb{Z}_p) \to \TR(R; \mathbb{Z}_p))$. 
A major advance of the work \cite{NS18}
is the insight that much of the ``glueing'' data that leads to the construction
of the fixed points $\THH(R)^{C_{p^n}}$ is actually redundant (and is not
needed to construct $\TC$ in particular). On the other hand, 
$\TR$ has recently been used to give an entirely new formulation (and ``decompletion'') of the theory of
cyclotomic spectra via the theory of topological Cartier modules due to
Antieau--Nikolaus \cite{AN};
this has many advantages, including that it yields a natural $t$-structure on
cyclotomic spectra. 

The most fundamental calculation in topological Hochschild homology is that of
$\mathbb{F}_p$. 
\begin{theorem}[B\"okstedt] 
\label{Bthm}
We have $\THH_*(\mathbb{F}_p) = \mathbb{F}_p[\sigma]$ with $|\sigma| = 2$. 
\end{theorem} 

For a discussion of a proof of \Cref{Bthm} (and in particular the
multiplicative structure, which is crucial to everything that follows),
cf.~\cite{HN19}.  
In particular, \Cref{Bthm} is closely related to the result of Hopkins--Mahowald that 
the free $\mathbb{E}_2$-algebra over $\mathbb{S}$ with $p = 0$ is $H
\mathbb{F}_p$. 
Moreover, using B\"okstedt's theorem, one can in fact give a complete
description of $\THH(\mathbb{F}_p)$ as a cyclotomic spectrum,
cf.~\cite[Sec.~IV-2]{NS18}. 

It is instructive to compare B\"okstedt's theorem with the calculation of
$\HH(\mathbb{F}_p/\mathbb{Z})$. 
One sees easily (e.g., using the Hochschild--Kostant--Rosenberg theorem,
cf.~\cite[Prop.~IV.4.3]{NS18}) that 
$\HH_*(\mathbb{F}_p/\mathbb{Z})$ is 
the divided power algebra $\Gamma^*[\sigma]$ (for the same class $\sigma$ in
degree two); in particular, replacing
$\mathbb{Z}$ by $\mathbb{S}$ replaces the divided power algebra by the
polynomial algebra.

For any $\mathbb{F}_p$-algebra $R$, one has 
the formula
\[ \THH(R) \otimes_{\THH(\mathbb{F}_p)} \mathbb{F}_p \simeq
\HH(R/\mathbb{F}_p).  \]
Given the above description of $\THH_*(\mathbb{F}_p)$, one may view $\THH(R)$ as
a ``one-parameter deformation'' of $\HH(R/\mathbb{F}_p)$ along $\sigma$. 
Upon taking circle-fixed points and passing to associated gradeds, this
observation is ultimately connected to 
the fact that crystalline cohomology gives a one-parameter deformation of de
Rham cohomology in characteristic $p$ (along the parameter given by ``$p$'').

Connections between $\THH$ and arithmetic 
have been explored in 
\cite{Hes96, HM03, HM04, GH06}, which relate the homotopy groups of the fixed
points of $\THH$ to (various forms) of the de Rham--Witt complex. 
The first such result, in equal characteristic, gives a complete calculation of $\TR$ in the case of a
regular $\mathbb{F}_p$-algebra: 

\begin{theorem}[Hesselholt \cite{Hes96}] 
Let $R$ be a regular $\mathbb{F}_p$-algebra.\footnote{Recall that a regular
$\mathbb{F}_p$-algebra is ind-smooth, by N\'eron--Popescu desingularization.} 
Then there is an isomorphism 
$\TR(R; \mathbb{Z}_p)_* \simeq W \Omega_R^{\ast}$, where $W \Omega_R^{\ast}$ is the de
Rham--Witt complex of Bloch--Deligne--Illusie \cite{Ill79}. 
This isomorphism carries the operators $F, V$ on $\TR(R; \mathbb{Z}_p)$ to the similarly named
operator $F, V$ on $W \Omega_R^{\ast}$. 
\end{theorem}

In particular, one obtains a complete calculation of $\TC$ using the fixed
points of operator $F$ on the de Rham--Witt forms. 
In mixed characteristic, \cite{HM04} introduces analogs of the de Rham--Witt
complex, and \cite{HM03, HM04, GH06} discuss the connections between $\TR$ of 
smooth algebras in mixed characteristic and the de Rham--Witt complex. 
This in particular was used in \emph{loc.~cit.} to verify the Lichtenbaum--Quillen conjecture for
certain $p$-adic fields (prior to the general proof by Voevodsky--Rost). 
See also \cite{LW21} for a new approach to this calculation inspired by the methods of
\cite{BMS2}.

\section{The cotangent complex and its wedge powers}

The building blocks of all the constructions involved are the cotangent complex
and its wedge powers; these are the ``animations'' (for our purposes, left Kan
extensions) 
of the usual differential forms functors. We begin with a brief review. 
Fix a base ring $k$. 

\begin{construction}[Left Kan extension]
Let $\mathcal{C}$ be an $\infty$-category admitting sifted colimits. 
Let $\mathrm{Poly}_k$ be the category of finitely generated polynomial
$k$-algebras, and let 
$F: \mathrm{Poly}_k \to \mathcal{C}$ be a functor. 
Then one can construct the \emph{left Kan extension} or \emph{left derived
functor} $LF: \mathrm{Ring}_k \to
\mathcal{C}$, extending the functor $F$ to $\mathrm{Ring}_k$. 
Explicitly, we define $LF$ on all polynomial $k$-algebras (possibly on
infinitely many variables) by forcing $LF$ 
to commute with filtered colimits. 
Given an arbitrary $k$-algebra $R$, we can choose a simplicial resolution $
P_\bullet \to R$ where each $P_i$ is a polynomial $k$-algebra, and then 
$LF(R)  = | F(P_\bullet)|$. 
\end{construction} 

The above construction 
(in various forms classical, going back to Quillen) 
is a type of nonabelian left derived functor \cite[Sec.~5.5.8]{HTT}. 
More generally, we can express the 
above construction using the theory of animated rings. 
Let $\mathrm{Ani}(\mathrm{Ring}_k)$ be the $\infty$-category of animated
$k$-algebras 
(also called simplicial commutative $k$-algebras; we refer to \cite{CS} for a
discussion of this terminology). 
Then with $\mathcal{C}$ as above one has an equivalence of $\infty$-categories 
$\mathrm{Fun}_{\Sigma}(\mathrm{Ani}(\mathrm{Ring}_k), \mathcal{C} ) \simeq
\mathrm{Fun}( \mathrm{Poly}_k, \mathcal{C})$ between sifted-colimit preserving
functors $\mathrm{Ani}(\mathrm{Ring}_k) \to \mathcal{C}$ and functors
$\mathrm{Poly}_k \to \mathcal{C}$. 

\begin{definition}[The cotangent complex and its wedge powers] 
Let $R$ be a $k$-algebra. 
Then the \emph{cotangent complex} $L_{R/k} \in \D(R)$ is defined as the left
derived functor of the functor $R \mapsto \Omega^1_{R/k}$ of K\"ahler
differentials.\footnote{In principle, the functor $L_{\cdot/k}$ takes values
in $\mathcal{D}(k)$ as the argument varies, but with more effort (e.g., using
the $\infty$-category of pairs of an animated ring and a module over it) one can
construct the functor as stated.}   
Similarly, the wedge power $\bigwedge^i L_{R/k} \in \D(R)$ (for $i \geq 0$) is
defined as 
the derived functor 
of the functor $\bigwedge^i_R \Omega_{R/k}$ of differential $i$-forms; this can also be defined using the
Dold--Puppe nonabelian derived exterior
powers $\bigwedge^i: \D(R)^{\leq 0} \to \D(R)^{\leq 0}$
(cf.~\cite[Sec.~25.2.1]{SAG} for a modern account) applied to the
cotangent complex. 
\end{definition} 

We refer to  \cite[Tag 08P5]{stacks-project} for a comprehensive treatment of
the cotangent complex. 
A basic fact about the cotangent complex  is that it agrees with ordinary
differential forms not only for polynomial $k$-algebras, but more generally for
smooth $k$-algebras. The other fundamental tools are 
the transitivity sequence for  a sequence of ring maps $A \to B \to C$, which
yields a cofiber sequence in $\mathcal{D}(C)$,
\begin{equation}  L_{B/A} \otimes^L_B C \to L_{C/A} \to L_{C/B},  \end{equation}
and the base-change property 
$L_{B/A} \otimes_A A' = L_{B \otimes_A A'/A'}$ for a map $A \to A'$ of
$k$-algebras
such that $B, A'$ are $\mathrm{Tor}$-independent over $A$ (if they are not
$\mathrm{Tor}$-independent, one has to consider the derived tensor product $B
\otimes^L_A A'$ as an animated ring itself). 

\begin{example} 
Using the cofiber sequence and base-change, we find that if 
$B = A/r$ for $r \in A$ a nonzerodivisor, then $L_{B/A} \simeq (r)/(r^2)[1]$. 
\end{example} 

\begin{example} 
More generally, suppose $B = A/I$ for $I \subset A$ an ideal generated by a
regular sequence; this (and its generalizations) will be one of the primary
examples for us. 
Then $L_{B/A} = I/I^2[1]$, which is the suspension of a free $B$-module 
(cf.~\cite[Tag 08SH]{stacks-project}, \cite[III.3.2]{Ill71})
A consequence is that 
$\bigwedge^i L_{B/A} = \Gamma^i( I/I^2) [i]$, for $\Gamma^i$ the $i$th divided
power functor on flat $B$-modules. This is a consequence of the 
d\'ecalage isomorphism of \cite[Sec.~I.4.3.2]{Ill71} between $\bigwedge^i ( M[1])
= (\Gamma^i M)[i]$ for any $M \in \mathcal{D}(B)^{\leq 0}$. 
\end{example}

\begin{proposition} 
\label{perfectcot}
Let $R$ be a perfect $\mathbb{F}_p$-algebra. Then $L_{R/\mathbb{F}_p} =0 $. 
\end{proposition} 
\begin{proof} 
For every $\mathbb{F}_p$-algebra $S$, the Frobenius $S \to S$ induces zero on
$L_{S/\mathbb{F}_p}$; this follows by inspection in the case of a polynomial
$\mathbb{F}_p$-algebra and then follows in general by taking simplicial resolutions. 
The claim follows. 
\end{proof}

The key structural result in \cite{BMS2} used in defining the motivic
filtrations 
is the 
flat descent for the cotangent complex and its wedge powers. 
This was originally observed by Bhatt \cite{Bha12} and is treated further in
\cite[Sec.~3]{BMS2}.

\begin{theorem}[{\cite{Bha12, BMS2}}] 
\label{flatcotangent}
Let $k$ be a base ring. Then the construction $A \mapsto \bigwedge^i L_{A/k}$, as a functor
from $k$-algebras to $\mathcal{D}(k)$, satisfies flat descent. 
More generally, for any $k$-module $M$, the construction $A \mapsto \bigwedge^i
L_{A/k} \otimes_k M$ 
satisfies flat descent.
\end{theorem} 
This result is proved using the transitivity cofiber sequence for the cotangent complex, 
flat descent for modules, and an inductive argument on the degree $i$. 
As pointed out in \emph{loc.~cit.}, it remains open whether the above functors are flat
hypersheaves. 

\section{Quasisyntomic rings}
\label{sec:quasisyntomic}

A key insight in \cite{BMS2} is that to understand invariants such as $\THH$,
etc.~  it is extremely clarifying 
to work with very ``large'' (e.g., perfectoid) rings. 
In particular, one should make highly ramified extensions by adding lots of
$p$-power roots. 
This strategy is expressed using the quasisyntomic topology, a key construction
of \cite[Sec.~4]{BMS2}; all the filtrations of \emph{loc.~cit.} are defined on
quasisyntomic rings, and probably cannot be defined more generally (without
sacrificing convergence properties). 
This class of rings is also extremely useful in other contexts, including in the
prismatic Dieudonn\'e theory of Ansch\"utz--Le Bras \cite{ALB20}.

\newcommand{\qsyn}{\mathrm{QSyn}}
\begin{definition}[Quasisyntomic rings] 
A ring $R$ is \emph{quasisyntomic} if: 
\begin{enumerate}
\item $R$ is $p$-complete, and the $p$-power torsion in $R$ is bounded, i.e.,
annihilated by $p^N$ for $N \gg 0$.  
\item The cotangent complex $L_{R/\mathbb{Z}_p} \in \D(R)$ 
has the property that $L_{R/\mathbb{Z}_p} \otimes^L_R (R/p) \in \D(R/p)$ has
$\mathrm{Tor}$-amplitude in $[-1, 0]$. 
\end{enumerate}
We let $\qsyn$ denote the category of quasisyntomic rings. 
\end{definition} 

\begin{example}[Complete intersections] 
Let $(R, \mathfrak{m})$ be a $p$-complete local noetherian ring  with $p \in
\mathfrak{m}$. 
Then $R$ is quasisyntomic if $R$ is a complete intersection, i.e.,
if for any (or one) surjection $f: A \twoheadrightarrow \widehat{R}$ with $A$ complete
regular local, the kernel of $f$ is generated by a regular sequence. 
Indeed, a result of Avramov \cite{Av99} states that 
$R$ is a complete intersection if and only if $L_{R/\mathbb{Z}_p}$ has
$\mathrm{Tor}$-amplitude in $[-1, 0]$. 
Therefore, 
if $R$ is a complete intersection then $R$ is clearly quasisyntomic. 
\end{example} 

The idea is that quasisyntomic rings are those which behave like local complete
intersections at the level of the cotangent complex (which is enough to control
all Hochschild-type invariants). 
However, this class of rings includes many highly non-noetherian examples. 

\begin{example}[Perfect rings] 
Any perfect $\mathbb{F}_p$-algebra $R$ (i.e., one where the Frobenius is an
isomorphism) is quasisyntomic. 
In fact, we have that $L_{R/\mathbb{F}_p} = 0$ 
as we saw in  \Cref{perfectcot}; the transitivity cofiber sequence applied to
$\mathbb{Z}_p \to \mathbb{F}_p \to R$ thus implies that
$L_{R/\mathbb{Z}_p}$ is the suspension of a rank $1$ free $R$-module. 
\end{example} 

\begin{example}[Witt vectors of perfect rings] 
If $R$ is a perfect $\mathbb{F}_p$-algebra, then 
the ring of Witt vectors 
$W(R)$ is quasisyntomic. In fact, 
$W(R)$ is $p$-torsionfree and $W(R)/p \simeq R$; one thus obtains that
$L_{W(R)/\mathbb{Z}_p}$ vanishes $p$-adically, whence the claim. 
\end{example}

The class of perfect $\mathbb{F}_p$-algebras 
admits a remarkable generalization to mixed characteristic, namely the class of
integral perfectoid rings \cite[Sec.~3.2]{BMS1} (based on the notion of perfectoid Tate ring
introduced in \cite{Sch12, KL15, Fo13}).

\begin{definition}[Perfectoid rings] 
A $p$-adically complete ring $R$ is called \emph{perfectoid} if 
$R$ can be expressed as the quotient $W(R')/\xi$, where $R'$ is a perfect
$\mathbb{F}_p$-algebra, and $\xi \in W(R')$ is an element of the form $[a] +
pu$  
where $u \in W(R')$ is a unit and $a \in R'$; here $[\cdot]$ denotes the
Teichm\"uller lift.\footnote{Such elements $\xi$ are also called
\emph{primitive} or \emph{distinguished}; in the terminology of \cite{Prisms},
a perfectoid ring $R$ is the same data as the \emph{perfect prism} $(W(R'),
(\xi))$.} 
\end{definition} 

In the above, by replacing $R'$ by its $a$-adic completion, 
which does not change the quotient $W(R')/([a] + pu)$, we may 
 in fact assume that $R'$ is $a$-adically complete. 
Note that this in particular implies that $R/[a]$ is an $\mathbb{F}_p$-algebra, and
the Frobenius induces an isomorphism 
of $\mathbb{F}_p$-algebras, 
$R/[a]^{1/p} \xrightarrow{\sim} R/[a]$. 
This is in fact the essential feature of 
perfectoid rings: 

\begin{proposition}[{\cite[Lem.~3.10]{BMS1}}] 
Let $R$ be a ring such that there exists a nonzerodivisor $\omega \in R$
such that: 
\begin{enumerate}
\item $R$ is $\omega$-adically complete. 
\item $\omega^p \mid p $.  
\item  The Frobenius induces an isomorphism $R/\omega \xrightarrow{\sim}
R/\omega^p$. 
\end{enumerate}
Then $R$ is perfectoid. 
Conversely, if $R$ is perfectoid, then there exists an element $\omega$ such
that $\omega^p \mid p$, and for any such element, the Frobenius map $R/\omega
\to R/\omega^p$ is an isomorphism. 
\end{proposition} 

\begin{remark} 
Let $R = W(R')/\xi$ be a perfectoid ring with $\xi = [a]  + p u$. 
By $a$-adically completing $R'$ if necessary, we may assume that $R'$ is
$a$-adically complete. 
In this case, $R'$ is the \emph{tilt} $R^{\flat}$ of $R$, namely, 
$R' = \varprojlim_{\varphi} R/p$. In fact, $R/p = R'/a$, and for any perfect
$\mathbb{F}_p$-algebra $S$ which is $x$-adically complete,  we see that $S$
agrees with the inverse limit perfection of $S/x$. 
\end{remark}

\begin{remark} 
A perfectoid ring $R$ is quasisyntomic. 
Indeed, the $p$-complete cotangent complex $L_{R/\mathbb{Z}_p} =
L_{R/W(R^{\flat})}$ is  the suspension of a free $R$-module of rank $1$, since
$W(R^{\flat}) \to R$ is the quotient by a nonzerodivisor. 
\end{remark}

\begin{example} 
The $p$-adic completion of the ring $\mathbb{Z}_p[p^{1/p^\infty}]$ 
is
perfectoid. 
In fact, this ring can be written as the quotient of 
\[  W( \mathbb{F}_p[t^{1/p^\infty}])/( [t] - p) =
\left( \mathbb{Z}_p[u^{1/p^\infty}] \right)_{\hat{p}}/(u- p) .  \]
More generally, let $R$ be a $p$-torsionfree, $p$-adically complete
$\mathbb{Z}_p[p^{1/p^\infty}]$-algebra. 
Then $R$ is perfectoid if and only if the Frobenius induces an isomorphism 
$\phi: R/p^{1/p} \xrightarrow{\sim}
R/p$. 
\end{example} 

\begin{example} 
The ring $\mathbb{Z}_p[\zeta_{p^\infty}]_{\hat{p}}$ is
perfectoid. In this case, we can 
form the ring $\mathbb{Z}_p[ q^{1/p^\infty}]_{\widehat{(p,q-1) }} = W(
\mathbb{F}_p[\epsilon^{1/p^\infty}]_{\widehat{(\epsilon -1)}})$ (via $q = [\epsilon]$) and the element 
$[p]_q := \frac{q^p -1}{q-1} = 1 + q + \dots + q^{p-1}$. 
By considering the map $\mathbb{F}_p[\epsilon^{1/p^\infty}]_{\widehat{(\epsilon
- 1)}} \to \mathbb{F}_p, \epsilon \mapsto 1$, 
one checks that the coefficient of $p$ in the Teichm\"uller expansion is a unit. 
Thus, we can take $R' = \mathbb{F}_p[\epsilon^{1/p^\infty}$ and $\xi = [p]_q$. 
\end{example}

The original definition of a \emph{perfectoid field} was given in \cite{Sch12}: a
perfectoid field $K$ is a complete nonarchimedean field $K$ with ring of
integers $\mathcal{O}_K \subset K$ such that the
valuation of $K$ is 
nondiscrete, $p$ is topologically nilpotent, and the Frobenius on $\mathcal{O}_K/p$ is surjective. 
This implies that there exists a nonzero topologically nilpotent element $\omega \in \mathcal{O}_K$ with $\omega^p 
\mid p $ and 
such that 
the Frobenius induces an isomorphism $\mathcal{O}_K/\omega \xrightarrow{\sim}
\mathcal{O}_K/\omega^p$. 
The datum of a perfectoid field is equivalent to the datum of a complete, rank
$1$ valuation ring which is perfectoid. 
The above two examples arise in this manner. 

 To obtain more examples of perfectoid rings, note that if $R$ is a
perfectoid ring, then the $p$-completion $R\left \langle
t^{1/p^\infty}\right\rangle$ of $R[t^{1/p^\infty}]$ is perfectoid. 
More subtly, there is the construction of the ``perfectoidization'' of a
semiperfectoid ring. 
If $R$ is a perfectoid ring and $I \subset R$ is a $p$-complete ideal, then
there is a $p$-complete ideal $J \supset I$ such that  $R/J$ is perfectoid, and
is the universal perfectoid ring to which $R/I$ maps
(cf.~\cite[Th.~7.4]{Prisms}). This construction is not easy to describe
explicitly in general. But if $I = (f)$ for $f \in R$ admitting a system of
$p$-power roots $\{f^{1/p^n}\}_{n \geq 0}$, then $J $ is the $p$-completion of
$\bigcup_{n \geq 0}
(f^{1/p^n})$.

We next review the quasisyntomic topology on $\qsyn$, cf.~\cite[Def.~4.1,
Cor.~4.8]{BMS2}. 
This is a non-noetherian version of the syntomic topology, cf.~\cite{FM87} or
\cite[Tag 0224]{stacks-project}, and in the $p$-complete context. 
Strictly speaking, it is $\qsyn^{op}$ that has the structure of a site. 

\begin{definition}[The quasisyntomic site] 
A map $R \to R'$ in $\qsyn$ is 
a cover if: 
\begin{enumerate}
\item $R/p^n \to R'/p^n $ is faithfully flat for all $n \geq 0$.   
\item  $L_{R'/R} \otimes^L_{R'} R'/p \in \D(R'/p)$ has $\mathrm{Tor}$-amplitude
in $[-1, 0]$. 
\end{enumerate}
\end{definition} 

The condition (1) is called \emph{$p$-complete faithful flatness}, and is the
appropriate replacement for faithful flatness in this (highly non-noetherian)
setup. Note that if $R$ is noetherian, then the condition (1) is simply faithful
flatness thanks to \cite{Yek18}. 

\begin{example}[Adding systems of $p$-power roots] 
Given a collection of elements $\{x_t \in R\}$, the ring $R'$ obtained as the
$p$-completion of 
$R[u_t^{1/p^\infty}, t \in T]/(u_t - x_t)$, i.e., obtained by
$p$-completely adding a system of $p$-power roots
of the elements $x_t$, gives a cover of $R$ in the quasisyntomic topology.  
Iterating this construction transfinitely many times, one sees that 
every object of $\qsyn$ can be covered by an object where all elements admit
compatible systems of $p$-power roots. 
\end{example} 

\begin{example}[Covers of regular rings] 
Let $R$ be a $p$-complete, regular noetherian ring. 
Then there is a quasisyntomic cover $R \to R_\infty$, with $R_\infty$
perfectoid. Conversely, a $p$-complete noetherian ring admitting such  a cover
is regular. This is proved in \cite[Theorem 4.7]{BIM}. 
\end{example}

\begin{example}[$p$-complete valuation rings are quasisyntomic] 
This follows by a result of Gabber--Ramero \cite[Th.~6.5.8]{GR03}. 
Moreover,  if $V$ is a valuation ring over $\mathbb{F}_p$, then
$L_{V/\mathbb{F}_p}$ is a flat $V$-module. 
\end{example} 

An important general 
structural result for perfectoid rings, formulated in terms of the quasisyntomic
site, is that locally one can add solutions to polynomial equations. 
This is highly non-trivial, since there is no obvious way to add such solutions
while retaining the perfectoid property. 
For this result, compare \cite[Sec.~2.5]{Andre18}, 
\cite[Th.~16.9.17]{GRfoundation}, 
\cite[Th.~7.12]{Prisms}, and \cite[Th.~2.3.4]{CS}. 

\begin{theorem}[Andr\'e's lemma] 
Let $R$ be a perfectoid ring. Then there exists a map of perfectoid rings $R \to
R_\infty$ such that: 
\begin{enumerate}
\item $R_\infty$ is absolutely integrally closed, i.e., every monic
polynomial equation over $R_\infty$ has a root in $R_\infty$.   
\item $R \to R_\infty$ is a cover in $\qsyn$. In fact, $R \to R_\infty$ can be 
taken to be the $p$-completion of an ind-syntomic map. 
\end{enumerate}
\end{theorem}

\begin{definition}[Quasiregular semiperfectoid rings]
A quasisyntomic ring $R$ is said to be \emph{quasiregular semiperfectoid} 
(or \emph{quasiregular semiperfect} if $R$ is additionally an
$\mathbb{F}_p$-algebra)
if
either of the following equivalent conditions hold: 
\begin{enumerate}
\item $R$ receives a surjection from a perfectoid ring.  
\item $R/p$ is semiperfect (i.e., the Frobenius is surjective), and $R$
receives a map from a perfectoid ring $R_0$. 
\end{enumerate}
To see that (2) implies (1), 
consider the map $\theta: W(R^{\flat}) \to R$ for $R^{\flat}$ the inverse limit
perfection of $R/p$; this is surjective modulo $p$,
hence surjective since $R$ is $p$-complete. 
The extension $R_0 \hat{\otimes}_{\mathbb{Z}_p} W(R^{\flat}) \to R$ is also
therefore surjective, and the source is perfectoid, whence (1). 
Note also that (1) implies (2) because the reduction mod $p$ of a perfectoid
ring is semiperfect. 
\end{definition}

We denote by $\qrsp$
the category of quasiregular semiperfectoid rings, equipped with the induced
site structure. 
The subcategory $\qrsp \subset \qsyn$ is a \emph{basis} for the quasisyntomic site: any 
object of $\qsyn$ admits a cover by an object of $\qrsp$. 
Moreover, the tensor product of two rings in $\qrsp$ remains in $\qrsp$.

\begin{remark} 
Suppose that $R$ is a quasiregular 
semiperfectoid ring. In this case, 
$L_{R/\mathbb{Z}_p}$ is the 
suspension of a $p$-completely flat $R$-module. 
\end{remark}

Heuristically quasisyntomic rings are those which behave like lci rings, at
least at the level of the cotangent complex (and after $p$-completion). We also
discuss a class of $\mathbb{F}_p$-algebras which behave more like smooth algebras.

\begin{definition}[Cartier smooth $\mathbb{F}_p$-algebras, \cite{KM18}] 
Let $R$ be an $\mathbb{F}_p$-algebra. We say that $R$ is \emph{Cartier smooth}
if: 
\begin{enumerate}
\item The cotangent complex $L_{R/\mathbb{F}_p}$ is a flat $R$-module in degree
zero. 
\item For each $i$, the inverse Cartier operator $C^{-1}:
\Omega^i_{R/\mathbb{F}_p} \to H^i( \Omega^{\ast}_{R/\mathbb{F}_p})$ 
is an isomorphism. 
Here the inverse Cartier operator is the unique map 
of graded algebras $\Omega^{\ast}_{R/\mathbb{F}_p} \to H^*(
\Omega^{\ast}_{R/\mathbb{F}_p})$ carrying $r \in R$ to the class of $r^p$ and
$ds, s \in R$ to the class of $s^{p-1} ds$. Compare \cite[Prop.~3.3.4]{BLM}. 
\end{enumerate}
\end{definition}

\begin{example} 
\begin{enumerate}
\item Any smooth algebra over a perfect field (or more generally a perfect
$\mathbb{F}_p$-algebra) is Cartier smooth, thanks to the classical Cartier
isomorphism (cf.~\cite[Th.~7.2]{Ka70} for an account).  
\item 
Any regular noetherian $\mathbb{F}_p$-algebra is Cartier smooth. Indeed, this
follows because Cartier smooth algebras are closed under filtered colimits and
any regular noetherian $\mathbb{F}_p$-algebra is ind-smooth by N\'eron--Popescu
desingularization. However, one can also prove this claim directly,
\cite[Sec.~9.5]{BLM}. 
\item Any valuation ring over $\mathbb{F}_p$ is Cartier smooth. 
This follows from results of Gabber--Ramero  
\cite[Th.~6.5.8]{GR03}
and Gabber \cite[App.~A]{KST20}. 
Conjecturally (by local uniformization, a weak form of resolution of singularities) valuation rings over $\mathbb{F}_p$ are
ind-smooth, which would imply Cartier smoothness, but local uniformization is not known in general. 
\item A collection of elements $\left\{x_i\right\}_{i \in I}$ in an
$\mathbb{F}_p$-algebra $R$ is a \emph{$p$-basis} if 
the elements $\prod_{i \in I} x_i^{a_i} \in R$, as $\left\{a_i \right\}_{i \in
I}$ ranges over all finitely supported functions $I \to \left\{0, 1,
\dots, p-1\right\}$, forms a basis for $R$ as a module over itself via the
Frobenius map. 
If $R$ admits a $p$-basis, then $R$ is Cartier smooth,
cf.~\cite[Th.~9.5.21]{BLM} and its proof. 
\end{enumerate}
\end{example}

We refer to \cite{KM18, KST20} for some applications of the theory of Cartier smooth
algebras. In particular, \emph{loc.~cit.} it is shown that the
 calculation 
 \cite{GL00, GH99}
 of the $p$-adic $K$-theory and topological cyclic homology
of regular local $\mathbb{F}_p$-algebras also generalizes to local Cartier smooth 
$\mathbb{F}_p$-algebras (e.g., valuation rings). 
\begin{remark} 
We do not know if the condition of Cartier smoothness guarantees that the
Frobenius endomorphism is flat. By a classical theorem of Kunz (see \cite{Ku69} or
\cite[Tag 0EC0]{stacks-project}), 
a noetherian $\mathbb{F}_p$-algebra is regular if and only if the Frobenius
endomorphism is flat. 
All the above examples of Cartier smooth algebras have the property that the
Frobenius is flat. 

On the other hand, condition (2) in the definition of Cartier smoothness is definitely not implied
by condition (1). For instance, there exist semiperfect $\mathbb{F}_p$-algebras
$R$ such that $L_{R/\mathbb{F}_p} = 0$ but such that $R$ is not perfect;
compare \cite{trivcc} for an example. 
In this case, the inverse Cartier operator reproduces the Frobenius $\varphi: R
\to R$ in degree zero, which is not an isomorphism. 
\end{remark} 

\begin{question} 
Is there an analog of Cartier smoothness for arbitrary quasisyntomic rings? 
\end{question} 
\section{Some quasisyntomic sheaves}

Throughout, we use the language of sheaves of spectra  \cite[Sec.~1.3]{SAG}
(this was implicitly used in the formulation of \Cref{flatcotangent}).
Note that this is slightly more general than the theory
introduced by Jardine \cite{Jardine}, which corresponds to the
subcategory of hypercomplete sheaves,
cf.~\cite{DHI}. However, all the sheaves used in the constructions of
\cite{BMS2} will be shown to be
hypercomplete (this is a convenient feature of the quasisyntomic site), so the
distinction does not play a significant role in \cite{BMS2}. 

\begin{definition}[Sheaves on $\qsyn$]
A spectrum-valued sheaf on $\qsyn$ is a functor $F: \qsyn \to \sp$ such that 
\begin{enumerate}
\item  
$F$ preserves finite products. 
\item
If $A \to B$ is a cover in $\qsyn$,
then the natural map 
$F(A) \to \varprojlim \left( F(B) \rightrightarrows F(B \hat{\otimes}_A B)
\triplearrows \dots \right) $ is an equivalence. 

\end{enumerate}
We let $\shv( \qsyn, \sp)$ denote the $\infty$-category of sheaves of spectra. 
\end{definition}

\newcommand{\shvh}{\mathrm{Shv}_{\mathrm{hyp}}}

\begin{definition}[Hypercomplete sheaves] 
We will say that a sheaf of spectra $F \in \shv(\qsyn, \sp)$ is \emph{hypercomplete} if 
$F$ satisfies descent for hypercovers in the quasisyntomic topology (rather than
only for \v{C}ech covers as above). 
We let $\shvh(\qsyn, \sp) \subset \shv(\qsyn, \sp)$ denote the subcategory of
hypercomplete sheaves; this inclusion is the right adjoint of a Bousfield
localization $(-)^h: \shv(\qsyn, \sp) \to \shvh(\qsyn, \sp)$ called
hypercompletion. 
\end{definition} 

The presentable, stable $\infty$-category $\shv(\qsyn, \sp)$ 
admits a canonical $t$-structure (as sheaves on any site do). 
A $\sp$-valued sheaf $\sF$ on $\qsyn$ is \emph{connective} if for every
$A \in \qsyn$ and $x \in \pi_j(\sF(A))$ for $j < 0$, there exists a quasisyntomic cover $A
\to B$ such that $x$ is carried to zero in $\pi_j( \sF(B))$. 
Similarly, $\sF$ is \emph{coconnective} if it takes values in coconnective
spectra. 
The $t$-structure restricts to a $t$-structure on the hypercomplete sheaves,
and every bounded-above sheaf is automatically hypercomplete.\footnote{The
hypercomplete sheaves are those sheaves which receive no maps from
$\infty$-connected sheaves.} 
With respect to this $t$-structure, the 
heart of $\shv(\qsyn, \sp)$ is the ordinary 
category of sheaves of abelian groups on $\qsyn$. 

\begin{construction}[Postnikov towers] 
Given any $\sF \in \shv(\qsyn, \sp)$, we have 
its Postnikov tower $\left\{\sF_{\leq n}\right\}_{n \in \mathbb{Z}}$ with
respect to the above $t$-structure. 
The limit of this Postnikov tower  is given by its hypercompletion $\sF^h$. 
This is a consequence of the fact that the quasisyntomic site is ``replete'' in
the sense of \cite[Sec.~3]{BSproet}; compare \cite[Prop.~A.10]{MTR}. In
particular, if $\sF$ is already hypercomplete, then $\sF$ is the limit of its
Postnikov tower. 
\end{construction}

A basic tool for working with sheaves on $\qsyn$ is restriction to the basis
$\qrsp \subset \qsyn$. 
In general, given a Grothendieck site, then it is a classical result
\cite[Exp.~III, Th.~4.1]{SGA4} that
sheaves of sets or abelian groups are equivalent to sheaves on any basis of the
site. 
The analog need not hold for sheaves of spaces or spectra, but it at least holds for
hypercomplete sheaves  in general, cf.~\cite[App.~A]{Aoki} or
\cite[Prop.~3.12.11]{Exodromy}. 
In the case of $\qrsp \subset \qsyn$, it is actually true that 
arbitrary sheaves on $\qsyn$ identify with sheaves on the basis $\qrsp$,
cf.~\cite[Prop.~4.31]{BMS2} or \cite[Lem.~C.3]{Hoy13} (for a more general
statement); the main point is that a pushout $B \widehat{\otimes}_A C$ in
$\qsyn$ along quasisyntomic covers with $B, C \in \qrsp$ belongs to $\qrsp$.  
Note that this strategy of restricting to $\qrsp \subset \qsyn$ is useful
precisely because we are working with such ``infinitary'' sites; it would be
much less useful if we worked with more classical sites such as the syntomic
or fppf site.

Now we discuss some examples of sheaves on $\qsyn$.

\begin{example} 
\label{cotangentwedgesheaf}
For each $i \geq 0$, the construction $R \mapsto \bigwedge^i
L_{R/\mathbb{Z}_p}[-i]$
defines a sheaf of spectra on $\qsyn$ (thanks to \Cref{flatcotangent}, with a slight
modification since we are working with $p$-completely faithful flatness). This sheaf  belongs to
the heart (so thus corresponds to a sheaf of ordinary abelian groups). 
In fact, this follows because it takes discrete values on the quasiregular
semiperfectoid rings. 
\end{example}

\begin{theorem} 
\label{descentforHH}
The functors
$\HH(-; \mathbb{Z}_p), \HH(-; \mathbb{Z}_p)^{h\mathbb{T}}, \HH(-;
\mathbb{Z}_p)^{t\mathbb{T}}$, 
$\THH(-; \mathbb{Z}_p),$ $ \TC^-(-; \mathbb{Z}_p),$ $ \TP(-;\mathbb{Z}_p), \TC(-;
\mathbb{Z}_p)$, etc. all
define hypercomplete sheaves on $\qsyn$. 
\end{theorem} 
In fact, all of these functors define (a priori not hypercomplete) sheaves on the ($p$-completely) flat
topology on all rings, as in \cite[Sec.~3]{BMS2}; for quasisyntomic rings the
argument shows that they are hypersheaves. 
One uses the Hochschild--Kostant--Rosenberg filtration \cite[Prop.~IV.4.1]{NS18}
to
prove that $\HH(-; \mathbb{Z}_p)$ is a hypercomplete sheaf on $\qsyn$ starting
from the fact that the $p$-complete cotangent complex and its wedge powers are sheaves 
on $\qsyn$. 
Taking homotopy fixed points, we find that $\HH(-; \mathbb{Z}_p)^{h\mathbb{T}}$ is a
hypercomplete sheaf. 
Similarly, using $\THH(-; \mathbb{Z}_p) \otimes_{\THH(\mathbb{Z})} \mathbb{Z} = \HH(-;
\mathbb{Z}_p)$ and taking the limit of the Postnikov tower of
$\THH(\mathbb{Z})$, one bootstraps to $\THH(-; \mathbb{Z}_p)$ and the invariants
defined from it.

\section{The motivic filtrations of \cite{BMS2}}

 To begin with, we describe the 
Hochschild--Kostant--Rosenberg filtration on Hochschild homology using the
quasisyntomic site. 

\begin{construction}[The Hochschild--Kostant--Rosenberg filtration] 
For a ring $R$ and any $R$-algebra $A$, there is a functorial,
complete multiplicative  $\mathbb{Z}_{\geq
0}$-indexed descending filtration 
$\mathrm{Fil}^{\geq \ast}_{\mathrm{HKR}} \HH(A/R)$
on $\HH(A/R)$ with $\mathrm{gr}^i \HH(A/R)
=\bigwedge^i L_{A/R}[i]$. This filtration is the Postnikov filtration when $A$
is a polynomial algebra over $R$ (using the Hochschild--Kostant--Rosenberg
theorem to identify the graded pieces), and is more generally 
defined via left Kan extension, cf.~\cite[Prop.~IV.4.1]{NS18}. 
A universal property of this filtration has been given by Raksit,
\cite{Raksit20}. 
\end{construction} 

The Hochschild--Kostant--Rosenberg filtration is the prototype of the motivic
filtrations of \cite{BMS2}. However, the strategy is to define the filtration by
descent from quasiregular semiperfectoids, i.e., by a right Kan extension
process rather than a left Kan extension process. These filtrations will
generally be more complicated to construct directly for polynomial algebras. 
To begin with, we show that the HKR filtration can be obtained for quasisyntomic
rings in such a fashion, after $p$-completion. 
\begin{construction}[The Hochschild--Kostant--Rosenberg filtration as a Postnikov
filtration] 
Suppose $R$ is a quasisyntomic ring. 
On the category of quasisyntomic $R$-algebras, we consider the functor $A
\mapsto \HH(A/R; \mathbb{Z}_p)$, which defines a hypercomplete sheaf of spectra. 
We claim that the homotopy sheaves 
are concentrated in even degrees, and that 
$ A\mapsto \mathrm{Fil}^{\geq \ast}_{\mathrm{HKR}} \HH(A/R; \mathbb{Z}_p)$
defines the double-speed Postnikov tower. 
In other words, $\mathrm{Fil}^{\geq i}_{\mathrm{HKR}}$ is the $2i$th connective
cover in quasisyntomic sheaves. 
This follows easily from the observation that if $A/R$ is such that the
$p$-completion $L_{A/R}$ is the suspension of a $p$-completely flat $R$-module, 
then $\mathrm{gr}^i \HH(A/R; \mathbb{Z}_p)$ is concentrated in degree $2i$
(e.g, if $A$ is a quasiregular semiperfectoid $R$-algebra), and
the HKR filtration reduces to the double-speed Postnikov filtration on the
individual spectrum 
$\mathrm{gr}^i \HH(A/R; \mathbb{Z}_p)$. 
\end{construction}

The starting point of the extension of the above strategy to invariants
defined from $\THH$ is the following generalization of B\"okstedt's theorem. 

\begin{theorem} 
\label{THHperfectoid}
 Let $R$ be a perfectoid ring. 
 Then one has an isomorphism 
 $\THH_*(R; \mathbb{Z}_p) \simeq R[\sigma]$, for $|\sigma| = 2$. 
 Moreover, one has $\pi_*(\THH(R; \mathbb{Z}_p)^{tC_p}) = R[u^{\pm 1}]$ for $|u| = 2$,
 and the Frobenius $\varphi: \THH(R; \mathbb{Z}_p) \to \THH(R;
 \mathbb{Z}_p)^{tC_p}$ exhibits the source as the connective cover of the
 target. 
\end{theorem} 

\Cref{THHperfectoid} reduces to B\"okstedt's theorem for $R = \mathbb{F}_p$, and
is extended to an arbitrary perfectoid ring in \cite[Sec.~6]{BMS2}. 
The case of $R = \mathcal{O}_{\mathbb{C}_p}$ had been previously proved in 
\cite{Hesselholt06}. 
See also \cite[Sec.~1.3]{HN19} for an account of this result. The basic strategy
is to bootstrap from the case $R = \mathbb{F}_p$, using the
Hochschild--Kostant--Rosenberg theorem and that for any map of
perfectoid rings $R \to R'$, the $p$-completed relative cotangent complex
$L_{R'/R}$ vanishes.

In \emph{loc.~cit.}, the 
constructions $\TC^-_*(R; \mathbb{Z}_p), \TP_*(R; \mathbb{Z}_p)$
are also identified. 
Let $A_{\mathrm{inf}}  = A_{\mathrm{inf}}(R)$ be the Witt vectors of
$R^{\flat}$, so one has a canonical surjection $\theta: A_{\mathrm{inf}} \to R$
with kernel generated by a nonzerodivisor $\xi \in A_{\mathrm{inf}}$. 
Then one has isomorphisms: 
\begin{gather} 
\TC^-_*(R; \mathbb{Z}_p) = A_{\mathrm{inf}}[x, \sigma]/(x \sigma = \xi) , \quad |x| = -2,
|\sigma| = 2\\
\TP_*(R; \mathbb{Z}_p) = A_{\mathrm{inf}}(R)[u^{\pm 1}], \quad |u| = 2.
\end{gather} 
Here $\sigma \in \pi_2 \TC^-(R; \mathbb{Z}_p)$ is a lift of the generator in
$\pi_2 \THH(R; \mathbb{Z}_p)$. 
With respect to these isomorphisms, the 
canonical map $\TC^-_*(R; \mathbb{Z}_p) \to \TP(R; \mathbb{Z}_p)$ 
carries $x$ to $ u^{-1} $ and $\sigma$ to $\xi u$. 
The 
cyclotomic Frobenius 
$\varphi: \TC^-_*(R; \mathbb{Z}_p) \to \TP_*(R; \mathbb{Z}_p)$ carries $\sigma
\mapsto u$ and $x \mapsto \phi(\xi) u^{-1}$, and is the Witt vector Frobenius on
$\pi_0$. 

Identifying $\TC^-_*(R; \mathbb{Z}_p), \TP_*(R; \mathbb{Z}_p)$ for quasiregular
semiperfectoids is significantly more difficult (and the description in purely
algebraic terms is a major result of \cite{BMS2, Prisms}). To begin with, we
make the simple observation that these are concentrated in even degrees. 
\begin{corollary}[Evenness for quasiregular semiperfectoids] 
\label{evennnessqrsp}
Let $A$ be a quasiregular semiperfectoid $R$-algebra. 
Then $\THH_*(A; \mathbb{Z}_p)$ is concentrated in even degrees. 
Consequently, $\TC^-_*(A; \mathbb{Z}_p), \TP_*(A; \mathbb{Z}_p)$ are concentrated in
even degrees. 
\end{corollary} 
\begin{proof} 
This follows from the equivalence
\begin{equation}  \THH(A; \mathbb{Z}_p) \otimes_{\THH(R; \mathbb{Z}_p)} R \simeq
\mathrm{HH}(A/R; \mathbb{Z}_p),   \label{THHdeformsHH} \end{equation}
the Hochschild--Kostant--Rosenberg filtration (which shows that the latter is
concentrated in even degrees). 
Then the $\mathbb{T}$-homotopy fixed point and Tate spectral sequences prove the
remaining claims. 
\end{proof} 
Similarly from \eqref{THHdeformsHH} one obtains: 
\begin{corollary} 
Let $A$ be a smooth algebra over the perfectoid ring $R$. Then 
$\THH_*(A; \mathbb{Z}_p) \simeq R[\sigma] \otimes_R
\Omega^{\ast}_{R/\mathbb{Z}_p}$ with $|\sigma| = 2$. 
\end{corollary} 
With respect to the above equivalence, the motivic filtration on $\THH(A;
\mathbb{Z}_p)$ is such that
$\sigma$ belongs to filtration $1$ and $\Omega^1_{A/R}$ belongs to filtration
$1$. This is not a Postnikov filtration, so it seems difficult
to construct the filtration on $\THH(A; \mathbb{Z}_p)$ purely within the setting
of smooth $R$-algebras. Thus, one needs to use instead the quasisyntomic site.

In particular, it follows from \Cref{evennnessqrsp} that the constructions $\THH(-; \mathbb{Z}_p), \TC^-(-; \mathbb{Z}_p),
\TP(-; \mathbb{Z}_p)$, when considered as objects of $\shv(\qsyn, \sp)$, have
homotopy groups concentrated in even degrees. In fact, the same holds for $\TC(-;
\mathbb{Z}_p)$. 

\begin{theorem}[{The odd vanishing conjecture, \cite[Sec.~14]{Prisms}}] 
\label{oddvanish}
The quasisyntomic sheaf $\TC(-; \mathbb{Z}_p)$ has homotopy groups concentrated
in even degrees. 
\end{theorem}

\Cref{oddvanish}
(which was conjectured in \cite{BMS2})
is much more difficult than \Cref{evennnessqrsp}. 
In particular, 
the evenness of $\TC(R; \mathbb{Z}_p)$ does not hold for an arbitrary
quasiregular semiperfectoid ring, and the proof relies on Andr\'e's lemma and
the theory of prismatic cohomology.

\begin{definition}[The motivic filtrations] 
The motivic filtration on $\THH(-; \mathbb{Z}_p)$ (resp.~$\TC^-(-;
\mathbb{Z}_p), $ $ \TP(-;
\mathbb{Z}_p), \TC(-;
\mathbb{Z}_p)$) is given as
the double speed Postnikov filtration in $\shv(\qsyn, \sp)$, in other words  
$\mathrm{Fil}^{\geq i} \THH(-; \mathbb{Z}_p) $ is the $2i$-th connective cover
of the quasisyntomic sheaf $\THH(- ; \mathbb{Z}_p)$.\footnote{In the
definition of the motivic filtration on $\mathrm{TC}(-; \mathbb{Z}_p)$, we want the formula \eqref{formulaTC} to work at the level of
filtered spectra, which here follows from the odd vanishing conjecture. One could also directly define the motivic filtration on $\mathrm{TC}(-;
\mathbb{Z}_p)$ to ensure this, which is the approach taken in \cite{BMS2}.  Then the fact that $\mathrm{Fil}^{\geq i}$
is the $2i$-connective cover (not simply the $(2i-1)$-connective cover) requires the
odd vanishing conjecture.} 
We define the objects
for $A \in \qsyn$,
\begin{gather}   \Prismc_A\left\{i\right\} = \mathrm{gr}^i \TP(A;
\mathbb{Z}_p)[-2i],  \\
\mathcal{N}^{\geq i} \Prismc_A\left\{i\right\} = \mathrm{gr}^i \TC^-(A;
\mathbb{Z}_p)[-2i], \\ \mathbb{Z}_p(i)(A) = \mathrm{gr}^i \TC(A;
\mathbb{Z}_p)[-2i].
\end{gather}
\end{definition} 

All these define sheaves of $p$-complete, coconnective spectra on $\qsyn$. 

\begin{remark}[The motivic filtrations on $\qrsp$] 
A priori, the motivic filtrations are defined using the abstract theory of
sheaves of spectra, and the $t$-structure there. 
However, if $A \in \qrsp$, the motivic filtrations 
on $\THH(A; \mathbb{Z}_p), \TC^-(A; \mathbb{Z}_p), \TP(A; \mathbb{Z}_p)$
are very explicit: they are
simply the double-speed 
Postnikov filtrations on these individual spectra. 
In other words, when restricted 
to quasiregular semiperfectoid rings, the individual \emph{homotopy groups} 
of $\THH(-; \mathbb{Z}_p)$, $ \TC^-(-; \mathbb{Z}_p)$, $\TP(-; \mathbb{Z}_p)$ form
sheaves of spectra, cf.~\cite[Sec.~7]{BMS2}. 
In particular, for a quasiregular semiperfectoid ring $A$, we have
$\Prismc_A\left\{i\right\} = \pi_{2i} \TP(A; \mathbb{Z}_p)$; this is an
invertible module over $\Prismc_A = \pi_0 \TP(A; \mathbb{Z}_p)$. 
\end{remark}

Indeed, the object $\Prismc_A = \mathrm{gr}^0 \TP(A; \mathbb{Z}_p)$ (for $A$
quasisyntomic) is perhaps the most fundamental of all the above structures and
is closely related to prismatic cohomology \cite{Prisms}. 
Let us discuss some of the structure that it carries, which follows directly
from its definition. 

Let $A \in \qrsp$. 
The \emph{Nygaard filtration} 
on $\Prismc_A = \pi_0 \TP(A; \mathbb{Z}_p) = \pi_0 \TC^-(A; \mathbb{Z}_p)$ is the
filtration that comes from the homotopy fixed point spectral sequence. 
In particular, we define $\mathcal{N}^{\geq i} \Prismc_A = \pi_0 ( \tau_{\geq 2i}
\THH(A; \mathbb{Z}_p))^{h\mathbb{T}} \subset \Prismc_A$. 
This defines a descending, multiplicative, and complete filtration on $\Prismc_A$
such that $\mathcal{N}^{\geq i} \Prismc_A/\mathcal{N}^{\geq i+1} \Prismc_A =
\pi_{2i} \THH(A; \mathbb{Z}_p)$. By descent, we obtain the Nygaard filtration on
$\Prismc_A$ for all quasisyntomic $A$. 
For $A$ quasiregular semiperfectoid, $\Prismc_A\left\{i\right\} = \pi_{2i} \TP(A;
\mathbb{Z}_p)$ is an invertible $\Prismc_A$-module (which can be trivialized, but
not canonically in general) 
and the notation above $\mathcal{N}^{\geq i}\Prismc_A\left\{i\right\} = \pi_{2i} \TC^-(A;
\mathbb{Z}_p) $ is consistent. 
The \emph{Frobenius} gives an endomorphism $\varphi: \Prismc_A \to \Prismc_A$
which for $A \in \qrsp$ comes from the cyclotomic Frobenius. 
The filtration and the Frobenius interact: we also have ``divided'' Frobenii
$\varphi_i : \mathcal{N}^{\geq i} \Prismc_A \left\{i\right\} \to
\Prismc_A\left\{i\right\}$ for $i \geq 0$, which arise from the cyclotomic
Frobenius on $\pi_{2i}$.

If $A$ is a quasiregular semiperfectoid algebra over the perfectoid ring $R$, 
then $\pi_{2i} \THH(A; \mathbb{Z}_p)$ has a finite filtration whose associated
graded terms are (the $p$-completions of) $\bigwedge^j L_{A/R}[-j]$ for $0 \leq
j \leq i$. Moreover, as $A$ ranges over quasiregular semiperfectoid $R$-algebras, we can
trivialize the Breuil--Kisin twists $\Prismc_A\left\{i\right\}$ for $i \in
\mathbb{Z}$ using the description of $\TC^-_*(R; \mathbb{Z}_p), \TP_*(R;
\mathbb{Z}_p)$. In particular, we have that 
$\varphi$ becomes divisible by $\phi(\xi)^i$ (one typically writes
$\widetilde{\xi} = \phi(\xi)$) on $\mathcal{N}^{\geq i} \Prismc_A$  and
we have a divided Frobenius 
$ \varphi/\widetilde{\xi}^i: \mathcal{N}^{\geq i} \Prismc_A \to \Prismc_A$. By descent, we
obtain this structure for any quasisyntomic $R$-algebra. 

\begin{example} 
In the base case of the perfectoid ring $R$, we have 
$\Prismc_R = A_{\mathrm{inf}}$ and  $\mathcal{N}^{\geq i} \Prismc_R = \xi^i
A_{\mathrm{inf}}$. 
Given a quasiregular semiperfectoid $R$-algebra $A$, the ideal $(\xi)$ is
$\Prismc_A$ is well-defined (and is contained in $\mathcal{N}^{\geq 1}
\Prismc_A$); however, it depends on the choice of perfectoid ring $R$. 
On the other hand, the ideal $(\widetilde{\xi})  = (\varphi(\xi))$ is
well-defined purely in terms of $A$ without reference to $R$. 
In fact, it is the kernel of the map 
$\Prismc_A = \pi_0( \TP(A; \mathbb{Z}_p)) \to \pi_0 ( \THH(A;
\mathbb{Z}_p)^{tC_p})$. 
\end{example} 

In particular, by analyzing topological Hochschild homology 
and its homotopical structure, one obtains the above quasisyntomic sheaf of
rings, equipped with the Frobenius and filtration. This is a structure of great
interest to $p$-adic arithmetic geometry in mixed characteristic. 
For formally smooth algebras over a perfectoid ring, this agrees with the
construction of $A_{\mathrm{inf}}$-cohomology of \cite{BMS1} (and later
\cite{Prisms}). 

In the next couple of sections, 
we will discuss the situation in more detail in characteristic $p$, where one
recovers the theory of crystalline cohomology.

\section{Derived de Rham cohomology}

In this section, we discuss some of the properties of $p$-adic derived de Rham
cohomology, after \cite{Bhattpadic}; see also \cite{SzZa15}. 
Fix a base ring $k$.

\begin{definition}[{Derived de Rham cohomology \cite[Sec.~VIII.2]{Ill72}}] 
Let $R$ be a $k$-algebra. 
The \emph{derived de Rham cohomology} $L \Omega_{R/k} \in \mathcal{D}(k)$ is the
left derived functor of the functor
$P \mapsto \Omega_{P/k}^\bullet$ sending a polynomial $k$-algebra to its de Rham
complex considered as an $\mathbb{E}_\infty$-algebra over $k$. 
Moreover, $L \Omega_{R/k}$ is equipped with the descending,
multiplicative derived Hodge filtration
$\left\{L \Omega_{R/k}^{\geq \ast}\right\}$
obtained as the left Kan extension of the naive filtration on the de Rham complex of a
polynomial $k$-algebra (i.e., the $i$th filtration piece consists of $j$-forms
for $j \geq i$). 
\end{definition}

\begin{remark}[The Hodge completion]
The Hodge completion of derived de Rham cohomology is often more tractable. 
For example, for a smooth $k$-algebra $R$, the Hodge completion of $L
\Omega_{R/k}$ agrees with the usual de Rham complex; this follows 
by considering the map from derived to underived de Rham cohomology (with
respective Hodge filtrations), and using that $\bigwedge^i L_{R/k} = \Omega^i_{R/k}$
for $R/k$ smooth. 
\end{remark} 
\begin{example}[Derived de Rham cohomology in characteristic zero] 
Let $k = \mathbb{C}$, and let $R$ be a finitely generated $\mathbb{C}$-algebra. 
On the one hand, derived de Rham cohomology of animated $\mathbb{C}$-algebras is easily
seen to be the constant functor with value $\mathbb{C}$ in this case; indeed,
this follows because the de Rham complex of a polynomial $\mathbb{C}$-algebra
is acyclic in positive degrees. 
On the other hand, the Hodge completion of derived de Rham cohomology agrees
with the singular cohomology (with $\mathbb{C}$-coefficients) of the associated
complex points. This is a classical result of Grothendieck \cite{Gro66} for $R$
smooth; compare \cite{Bha12} for a discussion in general. 
\end{example} 

In the sequel, we will only consider the $p$-adic version of derived de Rham 
cohomology, and we will simply drop the $p$-completion from the notation. 
We will also often drop the $p$-completion notation on the cotangent complex and
its wedge powers. 

\begin{construction}[The derived conjugate filtration] 
Let $A \to B$ be a map of animated $\mathbb{F}_p$-algebras. 
By left Kan extension of the Postnikov filtration (and using the Cartier
isomorphism) we see that $L \Omega_{B/A}$ admits a natural $B^{(1)} := B
\otimes_{A, \phi} A$-structure and an increasing, multiplicative, and
exhaustive filtration 
$\mathrm{Fil}_{\mathrm{conj}}^{\ast} L \Omega_{B/A}$ in $\mathcal{D}(B^{(1)})$; the associated graded pieces are given by 
$\mathrm{gr} ^i = \bigwedge^i L_{B/A}^{(1)}[-i]$. 
\end{construction} 

A key consequence of the derived conjugate filtration, the fact that
differential forms and the cotangent complex agree for smooth algebras, and the
Cartier isomorphism for smooth algebras, is the
following result. Note that it shows that derived de Rham cohomology behaves
entirely differently in characteristic $p$ than in characteristic zero. 
\begin{theorem}[{Bhatt \cite[Cor.~3.10]{Bhattpadic}}] 
\label{Lomegasmooth}
Given a smooth map $A \to B$ of rings, the 
$p$-complete derived de Rham cohomology $L \Omega_{B/A}$ agrees with the
$p$-complete de Rham
cohomology $\Omega^\bullet_{B/A}$ (with derived and classical Hodge filtrations
matching). 
\end{theorem}

\begin{remark} 
\label{CartsmoothddR}
Let $R$ be a Cartier smooth $\mathbb{F}_p$-algebra. 
Then the natural map 
$L \Omega_{R/\mathbb{F}_p} \to \Omega^\bullet_{R/\mathbb{F}_p}$ is an
equivalence respecting Hodge filtrations. This also follows from the conjugate filtration. 
In fact, for each $i$, the map 
$L ( \tau^{\leq i} \Omega_{R/\mathbb{F}_p}) \to \tau^{\leq i}
\Omega^\bullet_{R/\mathbb{F}_p}$ is an equivalence; one sees this on associated
graded terms, whence it follows from the assumptions. 
\end{remark} 

A further aspect of the $p$-adic theory is the appearance of certain $p$-adic
period rings when one applies $p$-adic derived de Rham cohomology to certain
large rings, shown in \cite{Bei12} in the Hodge-completed case and explored
further in \cite{Bhattpadic}. 
This phenomenon arises from the natural appearance 
of divided powers, cf.~\cite[Prop.~3.16]{SzZa15} for a detailed account. 
\begin{example}[Divided powers via derived de Rham cohomology] 
Consider the map $\mathbb{Z}_p[x] \to \mathbb{Z}_p$. 
Then the $p$-adic derived de Rham cohomology is given by the 
$p$-complete divided power algebra $\left( \mathbb{Z}_p\left[ \frac{x^i}{i!}
\right] \right)_{\hat{p}}$: more precisely, the natural map $\mathbb{Z}_p[x] \to
L \Omega_{\mathbb{Z}/\mathbb{Z}_p[x]}$ exhibits the target as the $p$-adic
divided power completion of $(x)$ in the source. 

To see this, we observe that everything involved has a grading. Formally, we work in the $\infty$-category of nonnegatively graded animated
rings $R_{\star}$ with $R_0 = \mathbb{Z}_p$. For any map $A \to B$ of such
nonnegatively graded animated rings, the construction $L \Omega_{B/A}$
carries through in this $\infty$-category, and it is not difficult to see that
the Hodge filtration converges for grading reasons (indeed, $L
\Omega_{B/A}^{\geq i}$ is concentrated in internal degrees $\geq i$ whence
$\varprojlim_i L \Omega_{B/A}^{\geq i} =0$ in the graded derived
$\infty$-category). 
In the graded $\infty$-category, 
the isomorphism
$L \Omega_{\mathbb{Z}_p/\mathbb{Z}_p[x]} = \left( \mathbb{Z}_p\left[ \frac{x^i}{i!}
\right] \right)_{\hat{p}}$ follows by passage to the associated graded 
of the Hodge filtration
$\mathrm{gr}^*( L \Omega_{B/A}) = \bigwedge^{\ast}
L_{B/A}[-\ast]$, using 
$L_{\mathbb{Z}_p/\mathbb{Z}_p[x]} = \mathbb{Z}_p[1]$ and the d\'ecalage isomorphism
$\bigwedge^i (M[1]) = \Gamma^i (M)[i]$.\footnote{One also uses here that if $A \to B$ is a map of
animated nonnegatively graded rings with $A_0 = B_0 = \mathbb{Z}_p$, then 
$A \to L \Omega_{B/A}$ is an isomorphism in degree $1$; this follows easily from
the case of a polynomial algebra.} 
By forgetting the grading, we conclude the
desired isomorphism. 
\end{example} 

In particular, if $A$ is a $p$-complete ring and $x \in A$ is a nonzerodivisor,
then the $p$-adic derived de Rham cohomology of $A \to A/x$ is simply the
$p$-complete divided power envelope of $(x)$; this follows from the above by
base-change.

\begin{construction}[Derived de Rham cohomology as a quasisyntomic sheaf]
Let $R$ be a quasisyntomic ring; for simplicity we assume $R$ is
$p$-torsionfree or an $\mathbb{F}_p$-algebra. On the category of quasisyntomic $R$-algebras, 
the construction $A \mapsto L \Omega_{A/R}$ defines a 
sheaf of spectra, which belongs to the heart of the $t$-structure
(in fact, it takes discrete values on quasiregular semiperfectoid algebras). This follows
from reducing modulo $p$ and the derived conjugate filtration. 
Similarly, $A \mapsto L \Omega_{A/R}^{\geq i}$ defines a sheaf of 
spectra (also in the heart). 
\end{construction} 
\begin{construction}[The Hodge-completed variant] 
Let $R$ be a quasisyntomic ring. On the category of quasisyntomic $R$-algebras, 
the constructions $A \mapsto \widehat{L \Omega_{A/R}}, \widehat{L
\Omega_{A/R}^{\geq i}}$ defines a 
sheaf of coconnective spectra, which belongs to the heart of the $t$-structure
(in fact, it takes discrete values on quasiregular semiperfectoid algebras). This follows
from the Hodge filtration. 
\end{construction}

The cohomology theories of \cite{BMS2}, for algebras over a perfectoid base, can
be described as ``deformations'' of (Hodge-completed) de Rham cohomology,
which therefore plays a central role in the theory. 
This arises as the combination of the following two results. 
The first (cf.~\cite[Sec.~5]{BMS2}) gives a close relationship between de Rham and periodic cyclic
homology. 
Other proofs (which work outside the $p$-complete context) have been given
by Antieau \cite{Ant18}, Moulinos--Robalo--To\"en \cite{MRT20}, and Raksit \cite{Raksit20}. 
For the result, we write $ \HC^- = \HH^{h\mathbb{T} }, \HP = \HH^{t
\mathbb{T}}$. 
\begin{theorem} 
\label{HPanddR}
Let $R$ be a quasisyntomic ring, and let $A$ be a quasisyntomic $R$-algebra
such that $L_{R/A}$ is the suspension of a $p$-completely flat module (e.g.,
$A$ could be quasiregular semiperfectoid). 
Then  we have natural isomorphisms
$$ \pi_{2i} \HP(A/R; \mathbb{Z}_p) = \widehat{L \Omega_{A/R}} , \quad 
\pi_{2i} \HC^-(A/R; \mathbb{Z}_p) = \widehat{L \Omega_{A/R}^{\geq i}}
.$$ 

In particular, by quasisyntomic descent, we 
obtain multiplicative, convergent exhaustive $\mathbb{Z}$-indexed
descending filtrations for any quasisyntomic $R$-algebra $A$,
on $\HC^-(A/R; \mathbb{Z}_p), \HP(A/R; \mathbb{Z}_p)$ with 
$$ \gr^i \HP(A/R; \mathbb{Z}_p) = \widehat{L \Omega_{A/R}}[2i], \quad 
\gr^i \HC^-(A/R; \mathbb{Z}_p) = \widehat{L \Omega_{A/R}^{\geq i}}[2i]
.$$
\end{theorem} 

\begin{remark} 
In characteristic zero and for $A/R$ smooth, the analogs of these filtrations are
canonically split (e.g.,
by Adams operations), and the connection 
between periodic cyclic and de Rham cohomology is classical,
cf.~\cite[Sec.~5.1.12]{Loday}. 
However, these filtrations are not canonically split in positive characteristic,
and the induced spectral sequences from de Rham cohomology to periodic
cyclic homology need not degenerate for smooth projective
varieties \cite{ABM}. 
\end{remark} 

The second result, which comes from analyzing the structure of $\THH(R;
\mathbb{Z}_p)$, 
states that $\TP$ gives a one-parameter deformation of $\HP$, for algebras 
over a perfectoid base. 

\begin{theorem}[{\cite[Th.~7.12]{BMS2}}] 
Let $R$ be a perfectoid ring, and let $A$ be any $R$-algebra. Then there is a
natural equivalence
\begin{equation} \label{TPmodxi} \TP(A; \mathbb{Z}_p) / \xi = \HP(A/R;
\mathbb{Z}_p).  \end{equation}
\end{theorem} 
More precisely, we have an equivalence of $\mathbb{E}_\infty$-algebras 
$\TP(A; \mathbb{Z}_p) \otimes_{\TP(R; \mathbb{Z}_p)} \HP(R/R; \mathbb{Z}_p)  =
\HP(A/R; \mathbb{Z}_p)$. 
Using \Cref{THHperfectoid} and the surrounding discussion, we have 
that $\TP_*(R; \mathbb{Z}_p) = A_{\mathrm{inf}}[u^{\pm 1}]$ and $\HP(R/R;
\mathbb{Z}_p) = R[u^{\pm 1}]$; the map $\TP_*(R; \mathbb{Z}_p) \to \HP_*(R/R;
\mathbb{Z}_p)$ has kernel generated by the element $\xi \in A_{\mathrm{inf}}$. 
Compare also \cite{AMN} for a discussion of related results.

By considering \eqref{TPmodxi} for $A$ a quasiregular semiperfectoid
$R$-algebra, 
combining with \Cref{HPanddR}, and using the definitions of the motivic
filtrations, we find that
\begin{equation} 
\label{PrismLO}
\Prismc_A / \xi = \widehat{L \Omega_{A/R}}.
\end{equation} 
By quasisyntomic descent, we obtain \eqref{PrismLO} for all quasisyntomic
$R$-algebras $A$. 
In particular, $\Prismc_A$ gives a one-parameter deformation of (Hodge-completed)
derived de Rham cohomology. 

\begin{remark}[Non-Nygaard complete prismatic cohomology]
Given a perfectoid ring $R$, 
one can define a ``Nygaard decompleted'' version $\Prism_{-}$ of $\Prismc_{-}$,
which deforms derived de Rham cohomology rather than its Hodge completion. Namely, one considers the quasisyntomic
sheaf $\Prismc_{-}$ and restricts to formally smooth $R$-algebras, and then left Kan extends
from formally smooth (or $p$-complete polynomial) $R$-algebras to all
$p$-complete $R$-algebras, as functors to $(p, \xi)$-complete
$\mathbb{E}_\infty$-algebras over $A_{\mathrm{inf}}$. 
This yields a construction 
$A \mapsto \Prism_{A}$ which provides a deformation along the parameter $\xi$ of
derived de Rham cohomology, i.e., one has functorial equivalences $\Prism_A/\xi
\simeq L \Omega_{A/R}$, which therefore also restricts to a sheaf on
quasisyntomic $R$-algebras (and belongs to the heart). 
At least a priori, this construction depends on 
the choice of the perfectoid ring $R$ mapping to $A$. However, in
\cite[Sec.~7]{Prisms}
a purely algebraic construction of $\Prism_{-}$ is given (on quasiregular
semiperfectoid rings, from which one can descend) that makes clear that
$\Prism_{-}$ can genuinely be defined on the whole quasisyntomic site, without
the choice of a perfectoid base.  
Similarly, $\Prism_{-}$ is still equipped with a Nygaard filtration
$\left\{\mathcal{N}^{\geq \ast}\Prism_{-}\right\}$ such that the completion with
respect to this filtration is $\widehat{\Prism_{-}}$; this follows because the
associated graded terms of the Nygaard filtration are left Kan extended from
$p$-complete polynomial rings as proved in \cite[Cor.~5.21]{AMNN}. 
\label{nygaarddecompletion}
\end{remark}

We have seen that $p$-adic derived de Rham cohomology coincides with the ``underived''
version for smooth algebras. 
More generally, there is a similar description in the case of a locally complete intersection
singularity (or a quasisyntomic ring) in terms of the divided
power de Rham complex of a polynomial algebra surjecting onto it. This fact is
essentially the comparison between crystalline cohomology and derived de Rham
cohomology \cite{Bhattpadic} and the classical description (due to Berthelot
\cite[Sec.~V.2.3]{Ber74}) of crystalline cohomology in terms of the
divided power de Rham complex, cf.~also \cite{BhattDJ} for another approach. 
We do not review the general theory of divided power structures in detail and give an ad
hoc construction; cf.~also the recent work \cite{Mao} for a treatment of the
derived divided power envelope construction. 

\begin{construction}[Divided power envelopes of free algebras] 
\label{divpowerenvelope1}
Let $(A, I)$ be a pair consisting of a $p$-torsionfree
$\mathbb{Z}_{(p)}$-algebra and an ideal $I \subset A$. 
Suppose that 
$A$ is a polynomial $\mathbb{Z}_{(p)}$-algebra and $I \subset A$ is the ideal
generated by a collection of the polynomial generators, 
i.e., $(A, I)$ is a free object (in the evident sense) in the category of such
pairs. 

We define the \emph{divided power algebra} $D_I(A)$ to be the subalgebra of $A
\otimes \mathbb{Q}$ generated by $A$ and the elements $\frac{y^i}{i!}, y \in I$;
this is also the divided power envelope (cf.~for instance \cite[Tag
07H7]{stacks-project}).
We have a descending multiplicative filtration
$\left\{\mathrm{Fil}^{\geq \ast} D_I(A)\right\}$ defined by the divided powers:
$\mathrm{Fil}^{\geq r} D_I(A)$ is the ideal generated by 
all elements $\frac{y_1^{j_1} \dots y_m^{j_m}}{j_1!\dots j_m!}$ for $j_1 + \dots
+ j_m \geq r$ for the $y_k \in I$. 
On the $A$-algebra $D_I(A)$, we have a flat connection 
$d: D_I(A) \to D_I(A) \otimes_A \Omega^1_{A/\mathbb{Z}_{(p)}}$ (extended from
$d: A \to \Omega^1_A$, so for instance
$d( \frac{y_1^i}{i!}) = \frac{y_1^{i-1}}{(i-1)!} dy_1$), and this connection
satisfies the Griffiths transversality property: $d( \mathrm{Fil}^{\geq r}
D_I(A)) \subset \mathrm{Fil}^{\geq r-1 } D_I(A) \otimes_A \Omega^1_A$. 
In particular, we can form the \emph{divided power de Rham complex}
\[ 
\Omega_{D_I(A)}^\bullet = 
D_I(A) \to D_I(A) \otimes_A \Omega^1_A \to D_I(A) \otimes_A \Omega^2_A \to
\dots ,  \]
and this is in turn equipped with a 
multiplicative filtration 
such that 
\begin{equation} \label{Hodgefiltdivpower}\mathrm{Fil}^{\geq r} \Omega_{D_I(A)}^\bullet = 
\mathrm{Fil}^{\geq r} D_I(A) \to \mathrm{Fil}^{\geq r-1} D_I(A) \otimes_A
\Omega^1_A  \to \mathrm{Fil}^{\geq r-2} D_I(A) \otimes_A \Omega^2_A \to \dots
.\end{equation}
\end{construction} 

Let $(A, I)$ be a free pair as above, and let $A_0, I_0$ denote the reductions
modulo $p$. We denote by $(-)^{(1)}$ the Frobenius twist along $A_0$, so
$A_0/\phi(I_0) = (A_0/I_0)^{(1)}$, for instance. 
Then one checks (cf.~\cite[Lem.~3.42]{Bhattpadic} and \cite[Prop.~8.11]{BMS2})
that 
\( D_I(A)/p  \) admits an ascending, exhaustive, multiplicative filtration (an analog of the
conjugate filtration)
such that $$\mathrm{gr}^0 D_I(A)/p =  A_0/\phi(I_0) = (A_0/I_0)^{(1)}$$ and in general
\begin{equation} \label{conjfiltdividedpow} \mathrm{gr}^i D_I(A)/p =
(\Gamma^i_{A_0/I_0}(I_0/I_0^2))^{(1)}. \end{equation}
Explicitly, the $i$th stage of the filtration 
on $D_I(A)/p$
is the $A_0/\phi(I_0)$-module
generated by $\frac{y_1^{j_1} \dots y_m^{j_m}}{j_1! \dots j_m!}$ for $j_1 +
\dots j_m \leq pi$. 

\begin{construction}[Derived divided power envelopes] 
Given any pair $(A, I)$ consisting of a $\mathbb{Z}_{(p)}$-algebra $A$ and an
ideal $I \subset A$, we can simplicially resolve the pair in terms of pairs $(B,
J)$ which are free in the sense above. 
Taking simplicial resolutions, we obtain the \emph{derived divided power
envelope} $L D_I(A)$ (a priori, an animated  ring) equipped with its
filtration $\mathrm{Fil}^{\geq \ast} LD_I(A)$ and a connection satisfying
Griffiths transversality. 
By left Kan extension of \eqref{conjfiltdividedpow}, there exists 
an analogous increasing, multiplicative, and exhaustive filtration on
$LD_I(A)/p$. 
\end{construction}

Suppose that $(A, I)$ is a pair such that $A, A/I$ are $p$-torsionfree, the
Frobenius on $A_0 = A/p$ is flat (e.g., $A/p$ is smooth over  a perfect ring), and such
that 
the $p$-completion of 
$L_{(A/I) / A} $ is $p$-completely flat. 
In particular, it follows from the conjugate filtration (and reducing modulo
$p$) that 
$LD_I(A)$ is a $p$-torsionfree, discrete ring. In particular, it is simply the
subring of $A \otimes \mathbb{Q}$ generated by the $\frac{y^i}{i!}, y \in I$,
and (by taking resolutions) one sees that it is actually the divided power
envelope in the usual sense \cite[Tag 07H7]{stacks-project}.  
We will only be interested in this case.

We now record the main result. 
Again, we emphasize that this result is essentially the comparison between
crystalline and derived de Rham cohomology as in \cite{Bhattpadic}. 

\begin{theorem}[$L \Omega$ via the divided power de Rham complex] 
\label{LOmegadivpower}
Suppose $(A, I)$ 
is a pair such that $A, A/I$ are $p$-torsionfree, $A/p$ is Cartier smooth and
has flat Frobenius, and such
that 
the $p$-completion of 
$L_{(A/I )/ A} $ is $p$-completely flat. 
Then there is a natural multiplicative, filtered isomorphism between the
$p$-adic derived de Rham cohomology $L
\Omega_{(A/I) / \mathbb{Z}_{(p)}}$ (with the Hodge filtration) and the $p$-completed divided power de Rham complex
$\Omega^\bullet_{D_I(A)}$ (with the filtration \eqref{Hodgefiltdivpower}). 
\end{theorem} 
\begin{proof} 
We will use a similar argument as in \Cref{Lomegasmooth}. In fact, 
it suffices to show that for a pair $(A, I)$ satisfying the above conditions, 
$\Omega^\bullet_{D_I(A)}$ is quasi-isomorphic modulo $p$ to its left Kan
extension from free pairs; for a free pair the natural map induces a filtered
quasi-isomorphism $\Omega^{\bullet}_{D_I(A)} \to
\Omega^\bullet_{(A/I)/\mathbb{Z}_{(p)}}$ by the Poincar\'e lemma (i.e., the
divided power Poincar\'e lemma; this is easy to check by hand,
cf.~\cite[Lemme~2.1.2]{Ber74}). For this, we will produce an appropriate filtration
on $\Omega^{\bullet}_{D_I(A)}/p$. 

By our assumptions, $D_I(A)$ is simply the subring of $A \otimes \mathbb{Q}$
generated by the divided powers of $I$. 
Thus, we can use the conjugate filtration $\mathrm{Fil}_{\mathrm{conj}}^{\leq
\ast} D_I(A)/p$, as in \eqref{conjfiltdividedpow}, which is defined by left Kan
extension from the case of a free algebra. From its definition (and left Kan
extension), we see that $A_0 = A/p$-connection on $D_I(A)/p$ is compatible with
the conjugate filtration. 
In particular, we have 
an ascending, exhaustive, multiplicative filtration on
$\Omega^\bullet_{D_I(A)}/p$ such that 
$\mathrm{gr}^i$ is the de Rham complex over $A_0$
of the $A_0$-module-with-connection $\Gamma^i_{A_0/I_0}(I_0/I_0^2)^{(1)}$. 
It thus suffices to show that this de Rham complex (where we write $I_0 = I/p$)
\begin{equation} \label{FrobdescdR} \Gamma^i_{A_0/I_0}(I_0/I_0^2)^{(1)} \to \Gamma^i_{A_0/I_0}(I_0/I_0^2)^{(1)}
\otimes_{A_0} \Omega^1_{A_0} \to \Gamma^i_{A_0/I_0}(I_0/I_0^2)^{(1)}
\otimes_{A_0}
\Omega^2_{A_0} \to \dots 
\end{equation}
as a functor from such pairs $(A, I)$ to $\mathcal{D}(\mathbb{F}_p)$, is left Kan
extended from the free objects. 
Now the $A_0$-connection  on $\Gamma^i_{A_0/I_0}(I_0/I_0^2)^{(1)}$ is the
canonical (Frobenius descent) connections, cf.~\cite[Lem.~3.44]{Bhattpadic}:
explicitly, this follows because any $ip$-th divided power $\gamma_{ip}(y), y
\in I$ is a flat section for this connection. 
Since this is a Frobenius descent connection, the Cartier isomorphism (valid
since $A_0$ is Cartier smooth) goes into
effect: the $j$th cohomology 
of \eqref{FrobdescdR} is given by
$\Gamma^i_{A_0/I_0}(I_0/I_0^2)^{(1)}\otimes_{A_0} (\Omega^i_{A_0})^{(1)}$. Thus, it
follows that \eqref{FrobdescdR} is left Kan extended from free pairs, whence the
result. 
\end{proof}

\newcommand{\acrys}{A_{\mathrm{crys}}}
\newcommand{\ainf}{A_{\mathrm{inf}}}

\section{The ring $\acrys$}
In this section, we will construct 
the functor $A \mapsto \Prismc_A$ together with the Nygaard filtration and
divided Frobenii in characteristic $p$. We describe the construction purely
algebraically here,  and in the next section will outline the proof that it is compatible with the
construction arising from topological Hochschild homology.

We first need the divided power construction for ideals containing $p$ (and
where the divided powers are compatible with the canonical divided powers on
$(p)$). 
The construction is analogous to that
of \Cref{divpowerenvelope1}; however, we will not have the analog of the Hodge filtration. 

\begin{construction}[Derived divided powers for ideals containing $p$] 
Let $(A, I)$ be a pair consisting of a $\mathbb{Z}_{(p)}$-algebra and an ideal
$I \subset A$ containing $p$. 

Suppose first $(A, I)$ is free: in other words, that $A$ is a polynomial ring and $I \subset A$ is the ideal
generated by a subset of the polynomial generators together with $p$. 
In this case, we define $D_I(A)$ to be the subring of $A \otimes \mathbb{Q}$
generated by $A$ and $\left\{\frac{y^i}{i!}\right\}_{y \in I}$. 
In general, we define the derived divided powers $L D_I(A)$ by 
simplicially resolving the pair $(A, I)$ by free objects, and taking the induced
simplicial resolution of divided power envelopes. 
\end{construction} 

As before, $L D_{I}(A)$ defines an animated ring, and its
rationalization is simply $A$. 
To control it in general, we 
again use the conjugate filtration. 
This gives that if $(A, I)$ is a free pair, then 
$D_I(A)/p$ has an ascending filtration as in \eqref{conjfiltdividedpow}. 
In particular, we find that if the pair $(A, I)$ is such that $A$ is
$p$-torsionfree, 
$A/p$ is Cartier smooth with flat Frobenius, and $L_{(A/I) / (A/p)} $ is the
suspension of a flat $A/I$-module, then $L D_I(A)$ is discrete and
$p$-torsionfree; we will thus simply write $D_I(A)$.  
In particular, again by taking resolutions and comparing, one verifies the
universal property that $D_I(A)$ is actually the divided power envelope of $(A,
I)$ compatible with the canonical divided powers on $(p)$.

\begin{definition}[The rings $\ainf, \acrys$] 
Let $R \in \qrspp$. 
\begin{enumerate}
\item  
The ring $R^{\flat}$ (the tilt of $R$) is defined as the inverse limit
perfection of $R$, i.e., $R^{\flat }  = \varprojlim_{\phi}R $. 
This comes with a natural map $R^{\flat} \to R$, and our assumption implies that
this map is surjective. 
\item
The ring $\ainf(R)$ is defined to be $W(R^{\flat})$; we have a natural
surjective map 
\begin{equation}  \theta: W(R^{\flat}) \to R .  \end{equation}
In particular, $\ainf(R)$ is the universal $p$-complete pro-nilpotent thickening
of $R$. 
\item 
The ring $\acrys(R)$ is defined as the $p$-complete (derived) divided power
envelope of the surjection $\theta $ (whose kernel includes $(p)$). 
\end{enumerate}
\end{definition} 

\begin{remark}[Properties of $\acrys$] 
Our assumptions imply that 
$LD_{\mathrm{ker \theta}} W(R^{\flat})$ is $p$-torsionfree by the conjugate
filtration. 
Therefore, $\acrys(R)$ can equivalently be obtained by taking the subring of
$\ainf(R)[1/p]$ generated by the divided powers of $\ker ( \theta)$, and then
$p$-adically completing again. 
In particular, it is actually a discrete $p$-torsionfree, $p$-complete ring (and
not an animated 
ring). 
Moreover, it is the $p$-completion of the (classical) divided power envelope of
$\theta$ compatible with divided powers on $(p)$, since the classical and
derived divided power envelopes coincide. 
\end{remark} 

\begin{remark} 
The choice of the map $\theta$ is in some sense arbitrary. 
Given any perfect $\mathbb{F}_p$-algebra $P$ with a surjection $P
\twoheadrightarrow R$, we could instead construct $\acrys(R)$ as the
$p$-complete divided
power envelope of $W(P) \twoheadrightarrow R$; this does not change the outcome. 
\end{remark} 

\begin{example} 
\label{perfectmodx}
Let $R$ be the ring $\mathbb{F}_p[x^{1/p^\infty}]/(x)$. 
Then $\acrys(R)$ is the $p$-adic completion of the subring 
$\mathbb{Z}_p [x^{1/p^\infty}, \frac{x^i}{i!}]_{i \geq 0} \subset
\mathbb{Q}_p[x^{1/p^\infty}]$. 
\end{example} 

The construction $R \mapsto \acrys(R)$ defines a sheaf of spectra (which
actually has image in discrete spectra) on $\qrspp$. 
Indeed, since $\acrys$ takes values in $p$-complete, $p$-torsionfree abelian
groups, it suffices to observe that $\acrys(R)/p$ defines a sheaf; but this in
turn follows from the conjugate filtration as in \eqref{conjfiltdividedpow} and descent for the
cotangent complex and its wedge powers. 
In fact, one can explicitly identify its reduction modulo $p$ as a sheaf on
$\qrspp$. 
One can give a proof of this using derived divided powers for
$\mathbb{F}_p$-algebras. 

We will be interested in the case of certain lci singularities, and first we
will need the following result.

\begin{theorem}[{Cf.~\cite[Prop.~8.12]{BMS2}}] 
\label{acrysmodpisLOmega}
For $R \in \qrspp$, we have a natural isomorphism
$\acrys(R)/p = L \Omega_{R/\mathbb{F}_p}$. 
\end{theorem} 

One can also prove the following closely related result, identifying $\acrys(R)$
with the derived de Rham cohomology of any $p$-adic lift. 
\begin{theorem} 
\label{acrysLOmega}
Let $S$ be  a quasisyntomic ring which is $p$-torsionfree and such that $R =
S/p$ is quasiregular semiperfect. 
Then we have a natural isomorphism 
$\acrys(R) = L \Omega_{S/\mathbb{Z}_p}$. 
\end{theorem} 
\begin{proof} 
Let $S^{\flat}$ denote the inverse limit perfection of $S/p$; we have a map 
$W(S^{\flat}) \to S$ which is surjective modulo $p$ by our assumptions, hence
surjective. 
By \Cref{LOmegadivpower}, it follows that $L \Omega_{S/\mathbb{Z}_p} = L
\Omega_{S/W(S^{\flat})}$ is the $p$-complete derived divided power envelope
of the surjection $W(S^{\flat}) \to S$. 
Similarly, by construction $\acrys(R)$ is the $p$-complete derived divided power 
envelope of the surjection $W(S^{\flat}) \to S/p$ compatible with the divided
powers on $(p)$. But it is easy to see that 
for any pair $(A, I)$ with $A, A/I$ $p$-torsionfree, the derived divided power
envelope of $(A, I)$ and the derived divided power envelope of $(A, (I, p))$
(where the latter is taken compatible with divided powers on $p$) agree:
indeed, this follows by left Kan extension from the polynomial case, when the
result is clear. 
\end{proof}

By descent, we obtain from $\acrys(-)$ a sheaf of spectra on
$\qsyn_{\mathbb{F}_p}$, which is a
$p$-adic lift  of the sheaf $L \Omega_{-/\mathbb{F}_p}$. 
One can show that this is precisely \emph{derived crystalline cohomology},
i.e., the functor on $\mathbb{F}_p$-algebras obtained by left Kan extending
(absolute) crystalline cohomology. 
In particular, the basic comparison theorems in crystalline (or de Rham--Witt)
theory yields that de Rham cohomology of smooth $\mathbb{F}_p$-algebras admits a $p$-adic lift 
given by crystalline cohomology. Left Kan extending to 
quasisyntomic $\mathbb{F}_p$-algebras, one obtains a $p$-adic lift of $L
\Omega_{-/\mathbb{F}_p}$, and one shows that this is precisely 
the above one. 
In other words: 

\begin{theorem}[{Cf.~\cite[Th.~8.14]{BMS2}}] 
For a quasiregular semiperfect $\mathbb{F}_p$-algebra $R$, 
there is a natural isomorphism between the derived crystalline
cohomology\footnote{In fact, one can also use the actual crystalline cohomology
of $R$, as one sees by the universal property of $\acrys$.} 
of $R$ (or the derived de Rham--Witt cohomology of $R$), obtained by left Kan
extending crystalline cohomology from polynomial $\mathbb{F}_p$-algebras, and
the ring $\acrys(R)$. 
\end{theorem} 

In particular, by descent from quasiregular semiperfect $\mathbb{F}_p$-algebras, 
the construction of the ring $\acrys(R)$ provides another approach to 
crystalline cohomology; of course, the approach is not essentially different
from the classical one,
since the definition of divided powers  is fundamental to the crystalline site
and the basic construction of crystalline cohomology. 
A key insight of \cite{BMS2} is that topological Hochschild homology provides  a
new (and fundamentally different) approach to the construction of crystalline
cohomology, where divided powers arise very naturally (instead of being
introduced by fiat). Most importantly, this approach has the advantage of working equally well in
mixed characteristic, where it reproduces the prismatic cohomology
\cite{Prisms}. 
Before formulating the results, we need one more ingredient. 

\begin{definition}[{The Nygaard filtration on $\acrys$, \cite[Sec.~8.2]{BMS2}}] 
\label{nygaardfiltacrys}
Let $R$ be a quasiregular semiperfect $\mathbb{F}_p$-algebra. 
We define the \emph{Nygaard filtration} $\{\mathcal{N}^{\geq \ast} \acrys(R) \}$  
such that $\mathcal{N}^{\geq i} \acrys(R) \subset \acrys(R)$ consists of those
elements $x \in \acrys(R)$ such that $p^i \mid \phi(x)$, where $\phi: \acrys(R)
\to \acrys(R)$ denotes the endomorphism induced by Frobenius. 
By construction, since everything involved is $p$-torsionfree, we have an
induced 
divided Frobenius map \[ \phi/p^i : \mathcal{N}^{\geq i} \acrys(R) \to
\acrys(R). \]
\end{definition}

\begin{example} 
Consider the case where $R = \mathbb{F}_p[x^{1/p^\infty}]/(x)$ as in
\Cref{perfectmodx},
so that $\acrys(R)$ is the $p$-completion of the ring
$\mathbb{Z}_p[x^{1/p^\infty}, \frac{x^j}{j!}]_{j \geq 0}$. 
Then $\mathcal{N}^{\geq i} \acrys(R)
$
is the $p$-completion of the subring generated by $p^{i-j}( \frac{x^j}{j!})$ for
all $j \geq 0$. 
\end{example} 

\begin{example}[The case of a $\delta$-lift] 
\label{deltaliftnygaard}
More generally, let $S$ be a $p$-torsionfree, $p$-complete $\delta$-ring 
such that $R = S/p$ is quasiregular semiperfect. Then we have the
canonical identification $\acrys(R) = L
\Omega_S$ (cf.~\Cref{acrysLOmega}), and 
the Nygaard filtration on $\acrys(R)$ is the tensor product of the Hodge
filtration on $L \Omega_S$ and the $p$-adic filtration on $\mathbb{Z}_p$. 
In other words, in the $p$-complete filtered derived category, we have
\begin{equation} \label{tensprodfiltLOmega} \left\{\mathcal{N}^{\geq \ast} \acrys(R)\right\}  = 
\left\{ L \Omega_S^{\geq \ast}\right\} \otimes \left\{p^{\ast}
\mathbb{Z}_p\right\}.
\end{equation}
This follows by left Kan extension from the case of $p$-completed free $\delta$-ring, using
\cite[Prop.~8.7]{BMS2} (or \cite[Prop.~8.3.3]{BLM} in the setting of
Dieudonn\'e complexes). 
By descent, we obtain \eqref{tensprodfiltLOmega} for arbitrary $p$-torsionfree
quasisyntomic $\delta$-rings. 
\end{example}

Although the definition of the Nygaard filtration in this manner does not make
the claim immediately evident, each $\mathcal{N}^{\geq i} \acrys(R)$ turns out
to define a sheaf of $p$-complete spectra on $\qrspp$. 
One way to see this is to use the following proposition, which gives very
strong control over the Nygaard filtration: 

\begin{proposition}[{Cf.~\cite[Th.~8.14(2)]{BMS2}}] 
\label{strongNygaardcontrol}
Let $R \in \qrspp$. 
For each $i \geq 0$, the map $\phi/p^i$ induces an isomorphism
\[ \mathcal{N}^{\geq i} \acrys(R)/\mathcal{N}^{\geq i+1} \acrys(R)
\xrightarrow{\sim} \mathrm{Fil}_{\mathrm{conj}, \leq i} ( \acrys(R)/p).   \]
\end{proposition} 

More precisely, $\phi/p^i$ (by construction) gives a well-defined map 
$\mathcal{N}^{\geq i} \acrys(R)/\mathcal{N}^{\geq i+1} \acrys(R)
\xrightarrow{\sim}  \acrys(R)/p$; the claim is that this map has image in 
the $i$th stage of the conjugate filtration (which is also the $i$th stage of
the de Rham conjugate filtration  under \Cref{acrysmodpisLOmega}), and induces an isomorphism
onto its image.

Since the Nygaard filtration gives a sheaf of $p$-complete spectra on $\qrspp$
 (thanks to \Cref{strongNygaardcontrol} and the conjugate filtration), by quasisyntomic descent we obtain a
filtration on the derived crystalline cohomology of quasisyntomic
$\mathbb{F}_p$-algebra (in particular any smooth algebra over a perfect field). 
This is the classical construction of the Nygaard filtration. 
To describe this, we review some fundamentals of de Rham--Witt theory, after
\cite{Ill79}. 

Suppose $R$ is a smooth algebra over a perfect field $k$. 
In this case, one has the de Rham--Witt complex 
$W \Omega_R^\bullet$ of 
\cite{Ill79}, which gives an explicit  $p$-torsionfree,
$p$-adically complete commutative differential graded
algebra representing the 
crystalline cohomology of $R$, 
\[ 
WR = W \Omega_R^0 \stackrel{d}{\to} W \Omega_R^1 \stackrel{d}{\to} W \Omega_R^2 \to \dots ,\]
and equipped with a map of commutative dg-algebras
$W \Omega_R^\bullet \to \Omega^{\bullet}_{R/k}$ which induces a
quasi-isomorphism $W \Omega_R^\bullet/p \to \Omega^\bullet_{R/k}$.
The object $W \Omega_R^\bullet$ is equipped with operators of graded abelian
groups $F, V: W \Omega_R^\bullet \to W \Omega_R^\bullet$ recovering the Witt vector
Frobenius and Verschiebung in degree zero and satisfying 
the identities $$FV = VF = p, \quad dF = pFd, \quad Vd = pdV,$$ (and $FdV = d$,
which is actually a consequence by
$p$-torsionfreeness). 
Another construction of $W \Omega_R^\bullet$ as a ``strict Dieudonn\'e algebra''
(and universal property/construction in that category) is given in \cite{BLM}. 

The endomorphism $\phi$ of the commutative differential graded algebra $W \Omega_R^\bullet$ induced by the Frobenius on $R$ 
is given in degree $i$ by $p^i F$. 
Then one can define a descending, multiplicative \emph{Nygaard filtration} on the
differential graded algebra $W \Omega_R^\bullet$
such that
\begin{equation}
\label{Nygaardpolynomial}
\mathcal{N}^{\geq i} W \Omega_R^\bullet = 
p^{i-1} V W\Omega_R^0 \to p^{i-2} V W \Omega_R^1 \to \dots \to V W
\Omega_R^{i-1} \to W \Omega_R^i  \to W \Omega_R^{i+1} \to \dots .
\end{equation}
By construction, the Frobenius $\phi$ (i.e., $p^d F$ in degree $d$) is divisible 
by $p^i$ on $\mathcal{N}^{\geq i} W \Omega_R^\bullet \subset W \Omega_R^\bullet$
and we can define a divided Frobenius (of cochain complexes) 
\[ \varphi/p^i: \mathcal{N}^{\geq i} W \Omega_R^\bullet \to W \Omega_R^\bullet.  \]
Moreover, one checks by an explicit homological calculation that the map of cochain complexes
\begin{equation} \label{divfrobgr}\varphi/p^i: \mathcal{N}^{\geq i} W \Omega_R^\bullet/\mathcal{N}^{\geq i+1} W
\Omega_R^\bullet\to  W \Omega_R^\bullet/p \to \Omega_{R/k}^\bullet
\end{equation}
induces a quasi-isomorphism 
from the source to the $i$-truncation of the target (the source
lives in degrees $\leq i$). 
All this can also be developed in the generality of strict Dieudonn\'e
complexes; compare \cite[Sec.~8]{BLM}. 

The Nygaard filtration is somewhat subtle in general. 
Indeed, a purely crystalline approach to the Nygaard filtration is not expected to exist
(in filtration degrees $\geq p$).

The basic comparison result is that 
the Nygaard filtration on $\acrys(R)$ recovers by descent the above Nygaard
filtration on crystalline cohomology. 

\begin{proposition}[{Cf.~\cite[Th.8.14(3)]{BMS2}}]
The Nygaard filtration $\left\{\mathcal{N}^{\geq \ast} \acrys(R)\right\}$
descends to the above Nygaard 
filtration (in the derived category) on crystalline cohomology for smooth
algebras over a perfect field. 
\end{proposition}

To summarize, from the above discussion, we have an equivalence between the following two
constructions of sheaves
on $\qsyn_{\mathbb{F}_p}$.
For convenience, we will use the first notation $L W \Omega_{-}$. 
\begin{enumerate}
\item The derived functor $R \mapsto LW \Omega_R$ of the derived de Rham--Witt
cohomology on polynomial (or smooth) $\mathbb{F}_p$-algebras, together with its
Nygaard filtration $\mathcal{N}^{\geq \ast} L W \Omega_R$ obtained by left Kan
extending the Nygaard filtration \eqref{Nygaardpolynomial} on polynomial rings. 
\item
The quasisyntomic sheaf $R \mapsto R \Gamma_{\qsyn}( \spec R, \acrys(-))$ given
on the basis $\qrspp$ by the construction $\acrys(-)$, equipped with the Nygaard
filtration $\{ R \Gamma_{\qsyn}( \spec R, \mathcal{N}^{\geq \ast} \acrys(-))\}$
obtained by descending the Nygaard filtration on $\acrys$ of quasiregular semiperfect
$\mathbb{F}_p$-algebras (\Cref{nygaardfiltacrys}). 
\end{enumerate}

In fact, 
either construction of the Nygaard filtration gives a map 
in the filtered derived category
for any $R \in \qsyn_{\mathbb{F}_p}$
\[ 
\phi: 
\mathcal{N}^{\geq \ast}L W \Omega_R \to p^{\ast} L W \Omega_R
\]
where the target denotes the $p$-adic filtration on 
the derived de Rham--Witt cohomology. 

We next need to discuss the completion of this sheaf with respect to the Nygaard
filtration. 
Again, there are two equivalent ways to proceed. The first is simply to take the completion
in the filtered derived category of $\left\{\mathcal{N}^{\geq \ast}LW
\Omega_R\right\}$ for any $R \in \qsyn_{\mathbb{F}_p}$, e.g., as constructed by
left Kan extension from polynomial algebras; we denote this by 
$\left\{ \mathcal{N}^{\geq \ast}\widehat{LW \Omega_R}\right\}$. 
The second (which we
describe below) is to take the Nygaard completion 
of $\acrys(-)$ for quasiregular semiperfect $\mathbb{F}_p$-algebras, and then
descend.

\begin{construction}[The Nygaard completion of $\acrys(-)$] 
Let $R \in \qrspp$. 
Then we define the \emph{Nygaard completion} $\acrysh(R)$
to be the completion of the ring $\acrys(R)$ with respect to the Nygaard
filtration, 
$\acrysh(R) = \varprojlim_i \acrys(R)/\mathcal{N}^{\geq i} \acrys(R)$. 
We also obtain the Nygaard filtration $\mathcal{N}^{\geq \ast} \acrysh(R)$
obtained by completion, and the completed divided Frobenius
$\phi/p^i: \mathcal{N}^{\geq i} \acrysh(R) \to \acrysh(R)$ for $i \geq 0$. 
\end{construction}

\begin{proposition} 
For $R \in \qrsp_{\mathbb{F}_p}$, the ring 
$\acrysh(R)$ is $p$-complete and $p$-torsionfree. 
Moreover, we have a natural isomorphism $\acrysh(R)/p \simeq \widehat{L
\Omega_{R/\mathbb{F}_p}}$ for $R \in \qrspp$. 
\end{proposition} 
\begin{proof} 
First, $\acrys(R)$ is clearly derived $p$-complete as an inverse limit of
modules of bounded torsion. 
The $p$-torsionfreeness follows because for $x \in \acrys(R)$, if $px \in
\mathcal{N}^{\geq i} \acrys(R)$, then $x \in \mathcal{N}^{\geq i-1} \acrys(R)$;
this is evident from the definition of the Nygaard filtration. 
Next, one verifies that the map $\acrys(R)/p \to L \Omega_{R/\mathbb{F}_p}$
(as in \Cref{acrysmodpisLOmega}) carries $\mathcal{N}^{\geq i } \acrys(R)$ to
the $i$th stage of the Hodge filtration $L
\Omega_{R/\mathbb{F}_p}^{\geq i}$, e.g., by calculating explicitly in the case
of  
$\mathbb{F}_p[x^{1/p^\infty}]/(x)$. The result then follows,
cf.~\cite[Th.~8.14]{BMS2}. 
\end{proof} 

When we descend from 
the quasiregular semiperfect $\mathbb{F}_p$-algebras, 
we find that $\widehat{L W \Omega_R}/p \simeq \widehat{L
\Omega_{R/\mathbb{F}_p}}$ for $R \in \qsyn_{\mathbb{F}_p}$, i.e., Nygaard
completed $LW \Omega_{-}$ gives a $p$-adic lift of Hodge-completed derived de
Rham cohomology. 
For smooth algebras over a perfect ring, it follows that 
the Hodge or Nygaard completion does nothing. More generally, by
\Cref{CartsmoothddR}, this also follows for Cartier smooth algebras: 

\begin{corollary} 
Suppose $R$ is a Cartier smooth $\mathbb{F}_p$-algebra (e.g., a smooth algebra
over a perfect field). 
Then the map $ LW \Omega_R \to \widehat{L W \Omega_R}$ is an equivalence, i.e.,
the Nygaard filtration is automatically complete. 
\end{corollary}

\begin{proposition} 
\label{Cartsmoothphidiv}
If the $\mathbb{F}_p$-algebra $R$ is Cartier smooth, then the map in the
filtered derived category
\[ \phi: \left\{\mathcal{N}^{\geq \ast} LW \Omega_R\right\} \to \left\{p^{\ast}
LW \Omega_R\right\}  \]
has the property that on $\mathrm{gr}^i$ the map is the truncation $\tau^{\leq i}$. 
(Equivalently, the map 
exhibits the source as the connective cover in the Beilinson $t$-structure,
cf.~\cite[Sec.~5]{BMS2}, of
the target.) 
\end{proposition} 
\begin{proof} 
Indeed, this follows because on associated graded terms, the divided Frobenius is identified with
the map 
$L ( \tau^{\leq i} \Omega^\bullet)_R \to L \Omega_{R/\mathbb{F}_p}^\bullet$, thanks to
\Cref{strongNygaardcontrol}  (and the
following discussion) by left Kan extension and descent. 
The hypothesis of Cartier smoothness implies that this map implements
$\tau^{\leq i}$-truncation, i.e., $L ( \tau^{\leq i} \Omega^\bullet_R) =
\tau^{\leq i} \Omega_R^\bullet$. 
\end{proof}

\section{The motivic filtrations for quasiregular semiperfect $\mathbb{F}_p$-algebras}

In this section, we describe the identification between the associated graded
pieces $\mathcal{N}^{\geq i} \Prism_R\left\{i\right\}$ of the motivic filtration
on $\TC^-(R; \mathbb{Z}_p)$ and the Nygaard-completed Nygaard pieces of
crystalline cohomology, for $R$ a quasisyntomic $\mathbb{F}_p$-algebra. 
By the usual descent we may work with $\qrspp$. 

Let $R \in \qrspp$ be a quasiregular semiperfect $\mathbb{F}_p$-algebra. 
From the cyclotomic spectrum $\THH(R)$, we obtain the following objects: 
\begin{enumerate}
\item The spectra $\TC^-(R) = \THH(R)^{h\mathbb{T}}, \TP(R) =
\THH(R)^{t\mathbb{T}}$, which have homotopy groups concentrated
in even degrees, calculated via the $\mathbb{T}$-homotopy fixed point and Tate
spectral sequences.   
\item An identification $\TP(R)/p  = \HP(R/\mathbb{F}_p)$ (a special case of
\eqref{TPmodxi}). 
\item The cyclotomic Frobenius $\varphi: \TC^-(R) \to\TP(R)$. 
\end{enumerate}

Consequently, given a quasiregular semiperfect $\mathbb{F}_p$-algebra $R \in \qrspp$, we
obtain the following: 
\begin{enumerate}
\item A $p$-adically complete, $p$-torsionfree ring $\widehat{\Prism_R}
\stackrel{\mathrm{def}}{=} \pi_0 (\TC^-(R))$ and an
endomorphism $\varphi_{\mathrm{cyc}}: \widehat{\Prism_R} \to
\widehat{\Prism_R}.$ induced by the cyclotomic Frobenius.   
\item A descending, multiplicative complete filtration $\mathcal{N}^{\geq \ast}
\widehat{\Prism_R}$ arising from the homotopy fixed point spectral sequence; in
particular, $\mathcal{N}^{=i } \Prism_R = \pi_{2i} \THH(R)$. 
Moreover, we have
$\mathcal{N}^{\geq i} \widehat{\Prism_R} = x^i \pi_{2i} \TC^-(R) \subset \pi_0
\TC^-(R)$. 
\item The property that $\varphi_{\mathrm{cyc}}( \mathcal{N}^{\geq i}
\widehat{\Prism}_R)  \subset p^i \widehat{\Prism}_R$. 
\item A canonical, multiplicative isomorphism $\widehat{\Prism}_R / p \simeq \widehat{L
\Omega_{R/\mathbb{F}_p}}$. 
\item When $R$ is perfect, $\widehat{\Prism_R} = W(R)$, 
the endomorphism $\phi_{\mathrm{cyc}}$ identifies with the (Witt vector)
Frobenius, the filtration $\mathcal{N}^{\geq \ast} \widehat{\Prism_R}$ is the
$p$-adic filtration. 
\end{enumerate}

\begin{theorem}[{Cf.~\cite[Th.~8.17]{BMS2}}] 
\label{TCandacrysh}
For $R \in \qrspp$, there is a functorial isomorphism of rings $\acrysh(R)
\simeq \widehat{\Prism_R}$, carrying the Nygaard filtration on $\acrysh(R)$ to
$\{\mathcal{N}^{\geq \ast} \widehat{\Prism}_R\}$. The cyclotomic Frobenius
$\varphi_{\mathrm{cyc}}$ on $\widehat{\Prism_R}$ agrees with the endomorphism
induced by the Frobenius $\varphi: R \to R$. 
\end{theorem} 

\begin{corollary}[The motivic filtration] 
For $R \in \qsyn_{\mathbb{F}_p}$, there is a convergent and exhaustive descending
$\mathbb{Z}$-indexed multiplicative
filtration $\mathrm{Fil}^{\geq \ast} \TC^-(R), \mathrm{Fil}^{\geq \ast} \TP(R)$ such that
\begin{gather} \mathrm{gr}^i \TC^-(R) = \mathcal{N}^{\geq i}
\widehat{LW \Omega_R}
[2i]  \\
\mathrm{gr}^i \TP(R) =  \widehat{L W \Omega_R}[2i] .
\end{gather}
With respect to these filtrations, the cyclotomic Frobenius is the divided
Frobenius $\varphi/p^i$ on the $i$th graded piece. 
\end{corollary}

We will give a proof of these results below, in a slightly different manner
than \cite{BMS2}. 
While there is an explicit topological argument in \cite{BMS2} in the case of
certain algebras, we argue instead using the following rigidity property of crystalline 
cohomology as a $p$-adic deformation of de Rham cohomology. 

\begin{theorem}[{Cf.~\cite[Th.~10.1.2]{BLM}}] 
\label{BLMrigidity}
Let $R \mapsto F(R),  \ \qrspp \to \mathrm{Rings}$ be a functor on quasiregular
semiperfect $\mathbb{F}_p$-algebras taking values in $p$-adically complete,
$p$-torsionfree rings. Suppose given an isomorphism of ring-valued functors
$F(-)/p \simeq L \Omega_{-/\mathbb{F}_p}$. 
Then there is a unique isomorphism $F(-) \simeq \acrys(-)$ lifting the specified
isomorphism modulo $p$. 
\end{theorem} 

\begin{remark} 
\Cref{BLMrigidity} has very recently been generalized by Mondal \cite{Mondal}: 
de Rham cohomology of $\mathbb{F}_p$-algebras admits a unique deformation over
any local artinian ring with residue field $\mathbb{F}_p$ (coming from the
base-change of crystalline cohomology). 
Mondal's work relies on some of the stacky ideas studied by Drinfeld
\cite{Drinfeld}. 
\end{remark} 

\begin{proof}[Proof of \Cref{TCandacrysh}] 
To begin with, we cannot directly apply \Cref{BLMrigidity} since 
for a quasiregular semiperfect $R \in \qrspp$, 
we have 
$\widehat{\Prism_R}/p = \widehat{L \Omega_{R/\mathbb{F}_p}}$, i.e., we obtain
the Hodge completion of the derived de Rham cohomology. 
We thus need to first ``decomplete;'' this will follow
\Cref{nygaarddecompletion} (in the case where the perfectoid base is
$\mathbb{F}_p$). 
To do this, 
we define the construction $R \mapsto \{ \mathcal{N}^{\geq \ast}\Prism_R\}$ on $\qsyn_{\mathbb{F}_p}$ by
restricting $R \mapsto \{ \mathcal{N}^{\geq \ast}\widehat{\Prism_R}\}$ to finitely generated polynomial $\mathbb{F}_p$-algebras
and then left Kan extending to all quasisyntomic $\mathbb{F}_p$-algebras. 
By construction (and \Cref{Lomegasmooth}), it follows that $R \mapsto \Prism_{R}$ is a $p$-adic
deformation of $R \mapsto L \Omega_{R/\mathbb{F}_p}$; in particular, $\Prism$
defines a sheaf of spectra on $\qsyn_{\mathbb{F}_p}$. 
It is easy to see that the completion of the filtered sheaf 
$\left\{\mathcal{N}^{\geq \ast} \Prism_R\right\}$ is indeed
$\left\{\mathcal{N}^{\geq \ast} \widehat{\Prism_R}\right\}$ (since the
associated graded terms of 
$\left\{\mathcal{N}^{\geq \ast} \widehat{\Prism_R}\right\}$, i.e., $\pi_{2*}
\THH(R)$, are already left Kan
extended from their unfolding to polynomial algebras). 

It follows from \Cref{BLMrigidity} that there is a unique functorial isomorphism
(for $R \in \qrspp$)
$\Prism_R \simeq \acrys(R)$ for $R \in \qrspp$ compatible with 
the isomorphism mod $p$ to $L \Omega_{R/\mathbb{F}_p}$; in particular, this
isomorphism is compatible 
the projection
maps to $R$. 
Moreover, again by left Kan extension the cyclotomic Frobenius defines an
endomorphism 
$$\varphi_{\mathrm{cyc}}: \Prism_R \to \Prism_R $$
carrying $\mathcal{N}^{\geq i}\Prism_R$ into $p^i \Prism_R$. 
We observe that $\phi_{\mathrm{cyc}}: \Prism_R \to \Prism_R$ (or $\acrys(R) \to
\acrys(R)$) is necessarily the endomorphism induced by functoriality
from the Frobenius $\varphi: R \to R$.
Note first that this is indeed the case when $R$ is perfect, by
\Cref{THHperfectoid}, and
$\Prism_R = \acrys(R) = W(R)$. 
It follows that when $R$ is semiperfect, 
we have by naturality (along $R^{\flat} \to R$) a commutative diagram
\[ \xymatrix{
W(R^{\flat}) = \acrys(R^{\flat}) \ar[d]  \ar[r]^{\phi_{\mathrm{cyc}}} &
W(R^{\flat}) \ar[d]  \\
\acrys(R) \ar[r]^{\varphi_{\mathrm{cyc}}} & \acrys(R)
}.\]
It follows that $\varphi_{\mathrm{cyc}}$ and $\acrys(\varphi)$ are both
endomorphisms of $\acrys(R)$ which agree when restricted to $W(R^{\flat})$; now
taking divided power envelopes and $p$-completing again show that they agree.

We have now shown that there is an isomorphism 
$\Prism_R \simeq \acrys(R)$, compatible with Frobenii (the cyclotomic Frobenius
and the Frobenius induced by functoriality), so we simply write $\varphi$. 
Now $\varphi( \mathcal{N}^{\geq i} \Prism_R)  \subset p^i \Prism_R$. 
It follows that under the above comparison, we have
$\mathcal{N}^{\geq i} \Prism_R \subset \mathcal{N}^{\geq i} \acrys(R)$
as submodules of $\Prism_R = \acrys(R)$: that is, the filtration coming from
$\mathrm{THH}$ is contained in the Nygaard filtration. It remains to show that
both filtrations are actually equal. Given this, it will follow that
$\widehat{\Prism_R} = \acrysh(R)$, compatible with filtrations and Frobenii (by
$p$-adic continuity).

It suffices to show that the inclusion
$\mathcal{N}^{\geq i} \Prism_R \subset \mathcal{N}^{\geq i} \acrys(R)$
is an equality for each $i$ and for the ring $R =
\mathbb{F}_p[x^{1/p^\infty}]/(x)$. 
Indeed, the inclusion will then be an equality for any tensor product of such
rings. Now any quasiregular semiperfect $\mathbb{F}_p$-algebra admits a
surjection from a tensor product $R' $of such rings which also induces a surjection
on cotangent complexes, from which we see that $\Prism_{R'} \to \Prism_R$ and
$\acrys(R') \to \acrys(R) $ induce surjections on filtered pieces
(cf.~\cite[Prop.~8.12]{BMS2} for this argument).  
Thus we can reduce to the case $R = \mathbb{F}_p[x^{1/p^\infty}]/(x)$.

Suppose $R = \mathbb{F}_p[x^{1/p^\infty}]/(x)$. 
we know that $\Prism_R = \acrys(R)$ is the $p$-completion of
$\mathbb{Z}_p[x^{1/p^\infty}, \frac{x^j}{j!}]_{j \geq 0}$. 
Indeed, it suffices (thanks to the explicit description of the Nygaard
filtration in this case, and since $p \in \mathcal{N}^{\geq 1} \Prism_R$) to see
that 
$\frac{x^i}{i!}$ belongs to the image of $ \mathcal{N}^{\geq i} \Prism_R \to
\Prism_R$ or equivalently maps to zero in the quotient
$\Prism_R/\mathcal{N}^{\geq i} \Prism_R$.  
For this, we use a grading argument, also used in \cite[Sec.~10.2]{BLM}. 
To this end, we can replace $R$ by $k[x^{1/p^\infty}]/(x)$ for $k$ 
a perfect field containing a transcendental element $t$; this admits an
automorphism given by sending $x^i \mapsto t^i x^i, i \in \mathbb{Z}[1/p]_{\geq
0}$. 
Now $\frac{x^i}{i!}$ has weight $i$ with respect to the induced grading (from
this automorphism), while
the weights of $\Prism_R/\mathcal{N}^{\geq i} \Prism_R$ are less than $i$ as one
sees from comparing $\mathcal{N}^{j} \Prism_R = \pi_{2j} \THH(R)$, which admits
a finite filtration by $\bigwedge^{j'} L_{R/\mathbb{F}_p}[-j']$ for $j' \leq j$.
\end{proof}

We now unwind the construction explicitly for Cartier smooth algebras and in particular verify the Segal conjecture.
In the context of topological Hochschild homology, the \emph{Segal conjecture} 
refers to the assertion that the cyclotomic Frobenius  $\varphi: \THH(R) \to
\THH(R)^{tC_p}$ should be an equivalence in sufficiently high degrees after
$p$-completion. The reason for the name comes from the case $R = \mathbb{S}$; in
this case, $\THH(\mathbb{S}) = \mathbb{S}$ and the Frobenius (or the unit map) $\mathbb{S} \to \mathbb{S}^{tC_p}$ is
actually a $p$-adic
equivalence \cite{Gun80, Lin80},\footnote{See also \cite{HW20} for a recent proof at
$p = 2$ using topological Hochschild homology.} which is the special case of the Segal Burnside ring
conjecture  for the group $C_p$ (in general a theorem of Carlsson \cite{Car84}). 
In the classical approach to topological cyclic homology, this implies that the
genuine $C_{p^n}$-fixed points of $\THH$ agree $p$-adically with the
$C_{p^n}$-homotopy fixed points for all $n \geq 0$, an insight due to
\cite{Tsalidis, BBLR} and formalized in \cite[Cor.~II.4.9]{NS18}. 
Many cases in which topological cyclic homology has been effectively computed for ring
spectra such as \cite{AR, Ausoni} are cases where the Segal conjecture holds,
and the Segal conjecture for $\THH$ seems to play a central role in the theory. 

Given this, it would be of interest to 
better understand the class of quasisyntomic rings 
for which some version of the Segal conjecture holds; this seems closely related
to some type of regularity of the ring. 
Since everything in sight is endowed with a motivic filtration, one expects to
see the Segal conjecture at the level of filtered pieces. 

We will describe this in characteristic $p$. 
First, if $R$ is any $\mathbb{F}_p$-algebra, we have an equivalence
of spectra $\THH(R)^{tC_p} \simeq \HP(R/\mathbb{F}_p) \simeq \TP(R)/p$,
cf.~\cite[Prop.~6.4]{BMS2}. 
Consequently, for $R \in \qsyn_{\mathbb{F}_p}$ we can define a complete,
$\mathbb{Z}$-indexed descending multiplicative motivic
filtration on $\THH(R)^{tC_p}$ such that $\gr^i = \widehat{L
\Omega_{R/\mathbb{F}_p}}[2i]$. 
With respect to this, the cyclotomic Frobenius $\phi: \THH(R) \to
\THH(R)^{tC_p}$ is a filtered map, and on $\gr^i$ it is given by the map 
$\phi/p^i: \mathcal{N}^{\geq i} \Prism_R/\mathcal{N}^{\geq i+1} \Prism_R \to
\Prismc_R/p = \widehat{L \Omega_{R/\mathbb{F}_p}}$; this follows from the
commutative diagram
\[ \xymatrix{
\TC^-(R) \ar[d]  \ar[r]^{\varphi} &  \TP(R) \ar[d]  \\
\THH(R) \ar[r]^{\phi} &  \THH(R)^{tC_p}
}\]
and descent from $\qrspp$, since the cyclotomic Frobenius realizes the divided
Frobenius. 
Note that $\TC^-(R)/x = \THH(R)$ for $x \in \pi_{-2}(\TC^-(\mathbb{F}_p))$ as
in \Cref{THHperfectoid} and the following discussion. 

The next result (for smooth algebras over a perfect field)  appears  in \cite[Cor.~8.18]{BMS2}. 
The last part had previously appeared in  \cite{Hes18}. 

\begin{corollary} 
For $R/\mathbb{F}_p$ Cartier smooth, 
the cyclotomic Frobenius $\THH(R) \to \THH(R)^{tC_p}$ has the property that on
$\mathrm{gr}^i[-2i]$,
it identifies with the $(-i)$-connective cover
$\tau^{\leq i} ( \Omega^{\ast}_{R/\mathbb{F}_p}) \to \Omega^{\ast}_{R/\mathbb{F}_p}$. 
If $\Omega^i_{R/\mathbb{F}_p} = 0$ for $i >d$ (e.g., $R$ could be smooth of
dimension $d$ over a perfect ring), then 
$\THH(R) \to \THH(R)^{tC_p}$ has $(d-3)$-truncated homotopy fiber. 
\end{corollary} 
\begin{proof} 
This follows from \Cref{Cartsmoothphidiv}, given the above discussion and
description of the map $\THH(R) \to \THH(R)^{tC_p}$ on associated graded pieces. 
Note that the fiber of the map on $\mathrm{gr}^i$ lives in degrees 
$\leq i-2$ for each $i$; under the second hypothesis,   
the map $\phi$ moreover induces isomorphisms on associated gradeds
$\mathrm{gr}^i$ for $i \geq d$, which yields the second assertion. 
\end{proof} 

\begin{question} 
Note that the results of \cite[Sec.~9]{BMS2} (as well as \cite{Prisms}) prove the Segal conjecture for 
smooth algebras over a perfectoid ring. 
The Segal conjecture does hold for $p$-complete regular noetherian rings (with
mild finiteness hypotheses), cf.~\cite[Sec.~5]{MTR} for an argument that
relies on the Beilinson--Lichtenbaum conjecture for the generic fiber. For
smooth algebras over the ring of integers in a $p$-adic field (for $p > 2$),
this was proved (in a purely $p$-adic fashion, using $\TR$) in \cite{HM03,
HM04}.  
Can one prove a filtered version of the Segal conjecture in such
cases?\end{question} 

\begin{remark} 
Another approach to the motivic filtration on $\TP$ in characteristic $p$ has
been given in \cite{AN}, using the expression $\TP(-; \mathbb{Z}_p) =
(\TR(-; \mathbb{Z}_p))^{t \mathbb{T}}$
proved in \emph{loc.~cit,} using that $\TR_*$ for a regular
$\mathbb{F}_p$-algebra recovers the de Rham--Witt complex \cite{Hes96}. 
This approach does not seem to recover the filtration in mixed characteristic,
though. 
\end{remark} 
\section{The $\mathbb{Z}_p(i)$: an example and some questions}

In this section, we revisit the calculation of the $p$-adic $K$-theory (or
equivalently the topological cyclic homology) of the dual
numbers over a perfect field $k$ of characteristic $p$. 
This calculation is due to Hesselholt--Madsen \cite{HM97}, using the methods of equivariant
stable homotopy theory, and has since been extended and generalized in various
directions (see for instance \cite{HM97trunc, Hes05, AGH09, AGHL}). The
calculation (more generally for truncated polynomial algebras) was
recently revisited by Speirs \cite{Speirs}, who gave another
approach using the Nikolaus--Scholze formula 
\eqref{formulaTC}
for $\mathrm{TC}$. 

To begin with, let us discuss some aspects of the motivic filtration on $\TC$ in
particular. 
\begin{construction}[The motivic filtration on $\TC(-; \mathbb{Z}_p)$]
Given an animated  $\mathbb{Z}_p$-algebra $R$, there is a natural complete, descending,
multiplicative $\mathbb{Z}_{\geq 0}$-indexed filtration $\mathrm{Fil}^{\geq
\ast} \TC(R; \mathbb{Z}_p)$ with \begin{equation} \label{grTC}\gr^i \TC(R; \mathbb{Z}_p) =
\mathbb{Z}_p(i)(R)[2i], \quad i \geq 0.\end{equation}
\end{construction} 

For quasisyntomic rings, the motivic filtration is defined using descent as
before: it is the double speed 
Postnikov filtration (in $\shv(\qsyn, \sp)$). In particular, for a quasisyntomic
ring $R$, we have the expression
\begin{equation} 
\label{Zpigeneralexpr}
\mathbb{Z}_p(i)(R)  = \mathrm{fib}( \mathrm{id}- \phi_i: \mathcal{N}^{\geq i}
\Prismc_{R}\left\{i\right\} \to \Prismc_{R}\left\{i\right\}).
\end{equation} 
This is a consequence of \eqref{formulaTC}, using that the two terms above are
as the graded quotients of $\TC^-(R; \mathbb{Z}_p), \TP(R; \mathbb{Z}_p)$. 
Now a key feature of $\TC(-; \mathbb{Z}_p)$ (not shared by $\TC^-(-;
\mathbb{Z}_p), \TP(-;\mathbb{Z}_p)$) is that it commutes with sifted colimits,
\cite[Theorem G]{CMM}. Moreover, one can check that the motivic filtration on
$\TC(-; \mathbb{Z}_p)$ when defined in the above manner on quasisyntomic rings
is actually left Kan extended from $p$-complete polynomial algebras, and
has the property that $\mathrm{Fil}^{\geq i} \TC(R; \mathbb{Z}_p) $ is
$(i-1)$-connective, cf.~\cite[Theorem 5.1]{AMNN}. Using these facts, one can left
Kan extend the motivic filtration on $\TC(-; \mathbb{Z}_p)$ from $p$-complete polynomial algebras to all
$p$-complete animated rings. This will agree with the previous definition on
quasisyntomic rings and, because of the connectivity statement, will converge.

Let us describe the $\mathbb{Z}_p(i)$ in characteristic $p$. 
\begin{construction}[The $\mathbb{Z}_p(i)$ in characteristic $p$] 
For $R $ an animated $\mathbb{F}_p$-algebra, we recall that we have the derived de
Rham--Witt cohomology $LW \Omega_R$, the Nygaard filtration $\mathcal{N}^{\geq i
}  $, and the divided Frobenius $\varphi/p^i: \mathcal{N}^{\geq i} LW \Omega_R
\to L W \Omega_R$. 
There is a natural isomorphism
\begin{equation} \mathbb{Z}_p(i)(R) = 
\mathrm{fib}(\mathrm{id} -\varphi/p^i): \mathcal{N}^{\geq i} L W \Omega_R \to LW
\Omega_R. \label{Zpicharp}
\end{equation}
In fact, for $R \in \qsyn_{\mathbb{F}_p}$, this expression after Nygaard completion is a consequence 
of \eqref{Zpigeneralexpr}
and the identification of $\Prismc_{-}$ and $\widehat{LW\Omega_{-}}$ on
$\qsyn_{\mathbb{F}_p}$. 
Then the main observation is that the Nygaard completion is
actually superfluous in the expression for the $\mathbb{Z}_p(i)$, by a $p$-adic
continuity argument, cf.~the proof of \cite[Prop.~8.20]{BMS2}.  
Left Kan extending from finitely generated polynomial algebras, we obtain
\eqref{Zpicharp} in general.  
\end{construction} 

Using the above, one can identify the $\mathbb{Z}_p(i)$ on regular
$\mathbb{F}_p$-algebras 
as the pro-\'etale cohomology of the logarithmic de Rham--Witt sheaves $W
\Omega_{\mathrm{log}}^i[-i]$, cf.~\cite[Cor.~8.21]{BMS2}; as in \cite{KM18}
this can be generalized to Cartier smooth algebras. 
By the results of \cite{GL00}, this shows that the $\mathbb{Z}_p(i)$ are in fact
$p$-adic \'etale motivic cohomology for regular $\mathbb{F}_p$-schemes. 

\begin{remark}[The Frobenius action] 
In general, if $R$ is any $\mathbb{F}_p$-algebra, 
a key feature of the $\mathbb{Z}_p(i)$ is that the Frobenius $\phi: R \to R$
acts as multiplication by $p^i$ on $\mathbb{Z}_p(i)(R)$; this is evident from
its expression \eqref{Zpicharp}. 
In particular, the motivic filtration on $\TC(R; \mathbb{Z}_p)$ becomes, after
rationalization, the eigenspace decomposition based on the Frobenius action, and
consequently splits canonically.
This is of course analogous to the motivic filtration on $K$-theory, which
rationally diagonalizes the Adams operations. 
\end{remark}

We will give here a description of the $\mathbb{Z}_p(i)$ of $k[x]/x^2$ (which
by the motivic filtration easily gives the description of
$\mathrm{TC}(k[x]/x^2)$); the
calculation is very close to that of \cite{Speirs}. 
The key observation is that $k[x]/x^2$ admits a natural lift to a quasisyntomic $\delta$-ring, 
so we can use the divided power de Rham complex to compute everything. 
Our strategy is to use 
the fact (cf.~\Cref{deltaliftnygaard}) that if $R$ is a $p$-torsionfree,
$p$-complete $\delta$-ring in $\qsyn$, then there is a functorial divided
Frobenius $\varphi/p^i : L \Omega_{R/\mathbb{Z}_p}^{\geq i} \to L
\Omega_{R/\mathbb{Z}_p}$, and the Nygaard
filtration on $L \Omega_{R/\mathbb{Z}_p} = LW \Omega_{R/p}$ is the tensor product of the  
$p$-adic filtration and the Hodge filtration. This is a consequence of the case
of $p$-complete polynomial $\delta$-rings (by a left Kan extension argument),
where it follows by a direct comparison between the de Rham complex of $R$ and
the de Rham--Witt complex of $R/p$. Compare \cite[Sec.~8.1.2]{BMS2}.

Now 
let $(A, I)$ be a pair 
such that $A, A/I$ are equipped with the compatible structure of $\delta$-rings, 
and suppose both $A, A/I$ are $p$-complete, $p$-torsionfree, and quasisyntomic. 
Suppose $A/p$ is Cartier smooth and the Frobenius on $A/p$ is flat. 
As we have seen, 
the Nygaard filtration on 
$  LW \Omega_{A/(I, p)} = L \Omega_{(A/I) / \mathbb{Z}_p}$ identifies with the tensor product of the
Hodge filtration 
and the $p$-adic filtration. 
Using 
\Cref{LOmegadivpower}, we can identify the 
de Rham complex (with its Hodge filtration, and Frobenius) of $A/I$ as the
divided power de Rham complex of $A$, with divided powers along $I$, and with
the divided power filtration. 
In light of this, we obtain an explicit cochain complex representing
$\mathbb{Z}_p(i)(A/(I, p))$. 

Indeed, we construct the divided power de Rham complex $\Omega^\bullet_{D_I(A)}$
with the (cochain-level) Nygaard filtration $\mathcal{N}^{\geq \ast}
\Omega^\bullet_{D_I(A)}$ (defined as the tensor product filtration as above). 
The Frobenius lift $\phi: A \to A$ induces a Frobenius lift $\phi$ on 
$\Omega^\bullet_{D_I(A)}$ which becomes divisible by $p^i$ on the
subcomplex\footnote{This crucially uses that $I$ is preserved by $\delta$.
For example, 
if $x \in I$, then the element $\frac{x^i}{i!} \in \mathrm{Fil}^{\geq i}D_I(A)$ has the
property that $\phi( \frac{x^i}{i!}) = \frac{(x^p + p \delta(x))^i}{i!} =
\sum_{a+b=i} \frac{x^{pa} p^b \delta(x)^b}{a!b!}$. Each term in the sum is
divisible by $p^i$ in $D_I(A)$.}
$\mathcal{N}^{\geq i} \Omega^\bullet_{D_I(A)}$, and we can take
\begin{equation} \mathbb{Z}_p(i)(A/(I, p)) = \mathrm{fib}\left( \mathcal{N}^{\geq i} \Omega^\bullet_{D_I(A)}
\xrightarrow{\phi/p^i-1} \Omega^\bullet_{D_I(A)} \right).
\label{Zpilambda}\end{equation}
To see this, we may reduce by descent to the case where $A/(I, p)$ is quasiregular
semiperfect and $A$ is a perfect $\delta$-ring (where the Frobenius $\phi$ is an
isomorphism), and then the above is effectively the definition. Thus, we get an 
expression for $\mathbb{Z}_p(i)(A/(I, p))$ as the mapping fiber of a cochain map
between explicit cochain complexes. 

In this section, we illustrate the above method by proving the following result. 
By the motivic filtration on $\TC$, this 
reproves the 
result of \cite[Th.~8.2]{HM97} (and by \cite{McCarthy97} yields the calculation
of $K_*(k[x]/x^2; \mathbb{Z}_p)$ since $K_*(k; \mathbb{Z}_p) = \mathbb{Z}_p$ in degree zero,
\cite{Hiller, Kratzer}). 

\begin{theorem} 
Let $k$ be a perfect $\mathbb{F}_p$-algebra for $p > 2$. 
Then for $i > 0$, we have
that $\mathbb{Z}_p(i)(k[x]/x^2)$ has no cohomology in degree outside $1$. 
Moreover, $H^1( \mathbb{Z}_p(i)( k[x]/x^2))$ is isomorphic to a direct sum 
$\bigoplus_{1 \leq d \leq 2i-1, (d, 2p) = 1} W_{n(i,d)}(k)$ where 
$n = n(i,d)$ is chosen such that $p^{n-1}d \leq 2i-1 < p^n d$. 
\end{theorem} 
\begin{proof}

In the above strategy, we consider the example $(A, I) = ( \widehat{W(k)[x]}_p, (x^2))$ where the $\delta$-structure is such
that $\delta(x) = 0$, so the Frobenius lift carries $x \mapsto x^p$. 
It follows 
from \Cref{LOmegadivpower}
that $LW \Omega_{k[x]/x^2}$ is given by the
$p$-completion of the divided power de
Rham complex
\begin{equation} \label{cryskxx2} W(k) \left[x, \frac{x^{2j}}{j!}\right]_{j \geq 0} \to W(k) \left[x,
\frac{x^{2j}}{j!}\right]_{j
\geq 0} dx.  \end{equation}
The Hodge filtration is (as in \Cref{LOmegadivpower} again), given by the divided power filtration and the
naive filtration.  That is, 
$ L \Omega_{W(k)[x]/x^2}^{\geq i}$ corresponds to the $p$-completion
of the subcomplex
\begin{equation} \label{divpowfiltkx}  \bigoplus_{j \geq i}
W(k) \left\{ \frac{x^{2j}}{j!}, 
\frac{x^{2j+1}}{j!}
\right \} _{j \geq i} \xrightarrow{d} \bigoplus_{j \geq i-1} W(k) \left\{
\frac{x^{2j}}{j!}, \frac{x^{2j+1}}{j!}\right\} dx
\end{equation}
Also, the Frobenius is defined on the complex by sending 
$x \mapsto x^p$. 

It follows now from 
\Cref{deltaliftnygaard}
that $\mathcal{N}^{\geq i} LW \Omega_{k[x]/x^2}$ is given by 
the $p$-completion of the complex 
\begin{equation} \label{Ngeqikx} \bigoplus_{j \geq 0}
p^{\mathrm{max}(i-j, 0)} W(k) \left\{ \frac{x^{2j}}{j!}, 
\frac{x^{2j+1}}{j!}
\right \} 
\xrightarrow{d} \bigoplus_{j \geq 0} 
p^{\mathrm{max}(i-j-1, 0)}
W(k) \left\{
\frac{x^{2j}}{j!}, \frac{x^{2j+1}}{j!}\right\} dx
.\end{equation}
In particular, thanks to \eqref{Zpilambda}, we obtain that $\mathbb{Z}_p(i)( k[x]/x^2)$ is the mapping fiber of 
$\phi/p^i - 1$ from the complex \eqref{Ngeqikx} to \eqref{cryskxx2}.

Let us evaluate the $\mathbb{Z}_p(i)(k[x]/x^2)$. 
First, we can evaluate the cohomology of $LW \Omega_{k[x]/x^2}$. 
Note that (other than the $W(k)$ in internal and cohomological degree zero) only $H^1$ is nonzero, and it has a natural grading (where $|x| = 1$);
with respect to this grading, we have
easily from \eqref{cryskxx2}
\[ \left( H^1( LW\Omega_{k[x]/x^2} )\right)_d = \begin{cases} 
W(k)/d & d \text{ odd } \\
0 & d \text{ even}
 \end{cases}. \]
Explicitly, the generator in degree $d  = 2j+1$ is  $\frac{x^{2j}}{j!} dx$. 

Similarly, from \eqref{Ngeqikx} we see that
$\mathcal{N}^{\geq i} LW \Omega_{k[x]/x^2}$ has cohomology concentrated in
(cohomological) degree
$1$ other than $p^i W(k)$ in internal and cohomological degree
zero,\footnote{Alternatively, we could phrase everything in terms of the
relative cohomology relative to the ideal $(x)$, and ignore these degree zero
terms; in any case they do not contribute to the $\mathbb{Z}_p(i)$ for $i > 0$.} and 
with respect to the internal grading we have
\begin{equation} \label{H1Ny} H^1( \mathcal{N}^{\geq i} LW \Omega_{k[x]/x^2})_d = 
\begin{cases}
W(k)/pd & {d = 2j+1, j < i } \\
W(k)/d & {d = 2j+1, j \geq i } \\
0 & d \text{ even} 
\end{cases} . \end{equation}
Explicitly, the generator in degree $2j+1$ is $p^{\mathrm{max}(i-j-1, 0 )}
\frac{x^{2j}}{j!} dx$. 

Now we need to understand the canonical and divided Frobenius maps. 
\begin{enumerate}
\item  
Consider 
$\varphi/p^i: \mathcal{N}^{\geq i} LW\Omega_{k[x]/x^2} \to LW\Omega_{k[x]/x^2}$. 
This map multiplies the internal grading by $p$. Suppose $j < i$.   Then it carries the class of 
$p^{i-j-1}
\frac{x^{2j}}{j!} dx$
into \begin{equation} p^{-i}p^{i-j-1} \frac{x^{p 2j} d(x^p)}{j!}  = 
p^{-j} \frac{x^{ (2j+1)p -1} }{j!} dx
.\label{classwhichisgen} \end{equation}
Now $p^j j!$, $(jp)!$, and $( \frac{(2j+1)p - 1}{2} )!$ agree up to  $p$-adic
units.  
It follows from this that 
for $d = 2j + 1$ for $j < i$, 
$\phi/p^i$ carries the degree $d$ summand 
$H^1( \mathcal{N}^{\geq i} LW\Omega_{k[x]/x^2})_d $ isomorphically to the degree
$pd$ summand of 
$H^1(  LW\Omega_{k[x]/x^2})_{pd} $ (in fact, it carries a generator to the
generator thanks to \eqref{classwhichisgen}). 
\item For $j \geq i$, 
the map $\phi/p^i$ carries the generator in degree $d  = 2j+1$ to a nonunit
multiple of the generator in degree $pd$. 
More precisely, the generator $\frac{x^{2j}}{j!} dx$ is carried to 
$p^{j-i+1}$ times a generator (arguing as above). 
\item
For $d = 2j+1$ for $j \geq i$, 
the canonical map induces an isomorphism on degree $d$ summands. 
\end{enumerate}

Fix an odd integer  $d \geq 1$ such that $d$ is not divisible by $p$. 
In this case, 
we consider the map of abelian groups
\begin{equation} \label{abgroupmap} \bigoplus_{a \geq 0} H^1( \mathcal{N}^{\geq
i} LW\Omega_{k[x]/x^2})_{p^a d} 
\xrightarrow{\phi/p^i - 1} \bigoplus_{a \geq 0}H^1(  LW\Omega_{k[x]/x^2})_{p^a d} 
\end{equation}
This suffices for the calculation of $\mathbb{Z}_p(i)(k[x]/x^2)$, since we can decompose the map $\phi/p^i -1$ over such $d$. 

Let $n = n(i, d)$  be such that $p^{n-1} d \leq 2i-1 < p^n d$. 
We claim that 
\eqref{abgroupmap} is surjective, and the kernel is $W_{n(i, d)}(k)$ if $d \leq
2i-1$ (and zero if $d > 2i-1$). 
The map 
\begin{equation} \label{abgroupmap2} \bigoplus_{a \geq n} H^1( \mathcal{N}^{\geq
i} LW\Omega_{k[x]/x^2})_{p^a d} 
\xrightarrow{\phi/p^i - 1} \bigoplus_{a \geq n}H^1(  LW\Omega_{k[x]/x^2})_{p^a d} 
\end{equation}
is seen to be an isomorphism after $p$-completion since the canonical map is an isomorphism while 
the divided Frobenius (when considered as an endomorphism via the inverse to the
canonical map) is locally nilpotent, thanks to the calculation in item (2). 
Thus, to prove the claim about \eqref{abgroupmap}, it suffices to quotient by
the summands for $a \geq n$, and to consider the map 
\begin{equation}
\label{abgroupmap3}
\bigoplus_{a < n} H^1( \mathcal{N}^{\geq i} LW \Omega_{k[x]/x^2})_{p^a d} 
\xrightarrow{\phi/p^i - 1} \bigoplus_{a < n}H^1(  LW\Omega_{k[x]/x^2})_{p^a d}. 
\end{equation}
In fact, the enumerated statements above show (e.g., by filtering 
both sides and passing to associated gradeds)
that 
$$\bigoplus_{a < n-1} H^1(\mathcal{N}^{\geq i} LW\Omega_{k[x]/x^2})_{p^a d} 
\xrightarrow{\phi/p^i - 1} \bigoplus_{a < n}H^1(  LW\Omega_{k[x]/x^2})_{p^a d}$$ is
an isomorphism, whence the kernel of \eqref{abgroupmap3}
is isomorphic to $H^1( \mathcal{N}^{\geq i} LW\Omega_{k[x]/x^2})_{p^{n-1}d} =
W_n(k)$ (by \eqref{H1Ny}), as desired. 
\end{proof}

It would be interesting to revisit the various calculations of topological
cyclic homology of $p$-adic rings, traditionally carried out using $\TR$ and
equivariant stable homotopy theory, 
using the motivic filtration (and in particular to calculate the
$\mathbb{Z}_p(i)$). In principle, 
the above gives a purely algebraic approach to the calculation for
$\mathbb{F}_p$-algebras with lci singularities. 

\begin{question} 
Can one calculate the $\mathbb{Z}_p(i)$ of $\mathbb{F}_p$-algebras with worse
than lci singularities? 
\end{question} 

A basic example
would be the case of a square-zero extension $k \oplus V$, for $k$ a perfect
field and $V$ a $k$-vector space, where the $K$-theory is calculated in
\cite{LM08}. This calculation has been extended to perfectoid rings by
Riggenbach 
\cite{Riggenbach}, using the approach to $\mathrm{TC}$ of \cite{NS18}.

In mixed characteristic, one knows \cite[Sec.~10]{BMS2} that for formally
smooth algebras over $\mathcal{O}_C$, the
$\mathbb{Z}_p(i)$ are given by the truncated $p$-adic nearby cycles of the usual
Tate twists on the generic fiber; they thus are closely related to integral
$p$-adic Hodge theory. For $i \leq p-2$, or when one works up to bounded
denominators, it is shown in \cite{AMNN} that 
that the $\mathbb{Z}_p(i)$ recover ``syntomic cohomology'' in a form essentially
due to \cite{FM87, Kato}. 
It would be interesting to carry out more calculations of the $\mathbb{Z}_p(i)$
 and $\TC$ in mixed characteristic. 

Finally, the $K$-theory of $\mathbb{Z}/p^2$
is only known in a limited range \cite{Bru01, Ang11}. 
\begin{question} 
Can one compute $\mathbb{Z}_p(i)(\mathbb{Z}/p^n)$ (and thus the $K$-theory of
$\mathbb{Z}/p^n$) for $n > 1$? 
\end{question} 

The work \cite{BCM}
uses prismatic cohomology to show that $L_{K(1)} K(\mathbb{Z}/p^n) = 0$ for $n
\geq 1$; this fact (and some generalizations) are also proved by different methods in
\cite{LMMT, MTR}. 
However, accessing the $p$-adic $K$-groups (or the $\mathbb{Z}_p(i)$) themselves seems to be substantially
more difficult. 
A stacky approach to prismatic cohomology has been proposed by Drinfeld
\cite{Drinfeld} and Bhatt--Lurie, and in particular one expects that the coherent cohomology of the object
$\Sigma''$ introduced in \emph{loc.~cit.} should be 
related to the $\mathbb{Z}_p(i)$. 
We hope that an increased understanding of the structure of $\Sigma''$ and of
prismatic cohomology in general
will also shed some light on these $K$-theoretic questions. 
The very recent work of Liu--Wang \cite{LW21} on calculating
$\TC(\mathcal{O}_K; \mathbb{F}_p)$ via descent-theoretic methods is an important step in
this direction. 

\bibliographystyle{amsalpha}
\bibliography{THHsurvey}

\end{document}